\documentclass[11pt,letterpaper]{amsart}

\usepackage[utf8]{inputenc}
\usepackage{amssymb, amscd, amsmath, epsfig, mathtools}
\usepackage{amsthm}
\usepackage{enumerate}
\usepackage{xcolor}
\usepackage{scalerel}
\usepackage{soul}
\usepackage{tikz-cd}
\usepackage{esint}
\usepackage{cancel}
\usepackage{hyperref}
\usepackage{slashed}
\usepackage{url}
\usepackage{booktabs}
\usepackage{algorithm}

\usepackage[margin=1.0in]{geometry}

\newtheorem{theorem}{Theorem}[section]
\newtheorem*{theorem*}{Theorem}
\newtheorem{definition}{Definition}[section]
\newtheorem{corollary}{Corollary}[section]
\newtheorem*{corollary*}{Corollary}
\newtheorem{lemma}{Lemma}[section]
\newtheorem{proposition}{Proposition}[section]
\theoremstyle{definition}
\newtheorem{remark}{Remark}[section]

\newcommand{\C}{\mathbb C}
\newcommand{\R}{\mathbb R}

\newcommand{\calC}{\mathcal C}

\newcommand{\calL}{\mathcal L}

\begin{document}

\title[Segal-Bargmann Transform in Neural PDEs]{The Segal-Bargmann Transform in Neural Partial Differential Equations}
\author{Jie Xu$^{*}$}
\thanks{$*$First author and the corresponding author.}
\address{Department of Technology, Operations and Statistics, New York University, New York, NY, USA}
\email{jx479@nyu.edu}
\author{Julio Castrill\'on-Cand\'as}
\address{Department of Mathematics and Statistics, Boston University, Boston, MA, USA}
\email{jcandas@bu.edu}
\author{Mark Kon}
\address{Department of Mathematics and Statistics, Boston University, Boston, MA, USA}
\email{mkon@bu.edu}
\author{Qinxun Bai}
\address{Isara Labs, San Francisco, CA, USA}
\email{jerry@isara.io}

\begin{abstract}
We introduce a novel neural operator method based on Segal-Bargmann transform and Bargmann-Fock space for learning the solution operators of partial differential equations (PDEs). The method brings data encoding, latent-space representation, operator learning, and decoding into an unifed framework. From only finitely many observations, we establish a sup-norm reconstruction error of order $ L^{-\frac{1}{2}} $, where L is the number of complex sampling points that are decoupled with the original samplings and observations. The reconstruction error is independent of the spatial dimension and the number of original sampling points. We show that this new mesh-free Segal-Bargmann Neural Operator approach has superior performance compared to existing NO approaches like FNO, BelNet, GNOT, etc.,\ particularly in terms of pointwise control $ \sup_{x}\lvert\hat{u}_{NO} - u_{truth}\rvert $ for solutions of PDEs $ \mathcal{L}u = f $ which provides pointwise accuracy and interpretability as classical numerical solvers. We also demonstrates the reconstructions of the solution operators for higher dimensional nonlinear PDEs, which suggest that our proposed method has great potential to be applied to more challenging PDEs and to even higher spatial dimensions.
\end{abstract}
\maketitle

\section{Introduction}
Let $ \Omega $ be an open subset of $ \R^{n} $. Consider the partial differential equation $ \mathcal{L} u = f \; {\rm in} \; \Omega \subset \R^{n} $ where $ u, f \in L^{2}(\R^{n}; \R) $ and $ L $ is some smooth differential operator, possibly with boundary and initial conditions. In this article, we are interesting the following neural operator learning question to study $ \mathcal{L}^{-1} $:
\medskip

{\bf{Question}}: Given only a finite collection of observations $ \{ u^{(k)}(x_{i}), f^{(k)}(x_{i}) \}_{k = 1}^{K} $ of unknown ground truth pairing of functions $ \{ u^{(k)}, f^{(k)} \}_{k = 1}^{K} $ on finitely many grids $ \{ x_{i} \}_{i = 1}^{N} \subset \R^{n} $, where each $ u^{(k)}(x) $ solves $ \mathcal{L}u^{(k)}(x) = f^{(k)}(x) $ with prescribed boundary conditions. 
Can we learn a neural operator $ G_{NO} $ that takes the sampled values $ \mathbf{f} = (f(x_{1}), \dotso, f(x_{N})) $ of a new input function $ f $ from the prescribed function class, and reconstruct $ \hat{u}(x) : = \mathcal{G}_{NO}(\mathbf{f})(x) $ that approximates the corresponding solution $ u $ of $ \mathcal{L}u = f $ at any query locations $ x \in \Omega $?
\medskip

The classical numerical solvers are designed to address this type of problem, but their computational cost can grow rapidly with spatial dimension; in addition, a new input instance generally requires a separate numerical solve of the underlying PDE. Over the last several decades, neural operator methods with deep learning architecture have been developed to learn the solution map for an entire family of PDE problems~\cite{DeepONet,NeuralOperator}. A summary of representative approaches is given in \S2. 

However, neural operator methods have several major issues: generalization remains a major limitation of neural operators, especially for out-of-distribution inputs, unseen discretizations, or different sampling grids; the curse of dimensionality is still not overcome; existing methods also lack the same level of stability guarantees, rigorous error control, and interpretability as classical numerical solvers. Moreover, a trained neural operator is usually tied to a specific PDE family and parameter regime, so substantial retraining may be required when the underlying equation or physical setting changes. For example, many numerical comparisons only report relative $ L^{2} $ errors, which alone do not control the error uniformly in space; the standard FFT-based implementation FNO \cite{FNO} require sampling over uniform grids, other methods require associated preferred patterns of grids; a full tensor-product Fourier representation requires exponentially many modes as the spatial dimension increases, although with a fixed cutoff in each coordinate; in the standard DeepONet architecture, the latent basis is learned implicitly by the trunk network rather than prescribed from an analytically interpretable function space, which can make the resulting representation difficult to interpret. 

In this article, we introduce a novel neural operator method based on the Segal-Bargmann transform, where the learning process is carried out in a new space with good properties. In addition, we also show that the combination of the Segal-Bargmann transform and the RKHS approximation serves as a good encoder and a finite dimensional representation of the original discrete data that are coming from evaluation representations of functions. In addition to the classical $ L^{2} $-type convergence rate, we give an $ L^{\infty} $-type convergence rate. 

The Segal-Bargmann transform was initially introduced by the independent work of Irving Segal~\cite{Segal1962} and
Valentine Bargmann~\cite{Bargmann} in the early 1960s, where Segal gave the general framework on Hilbert spaces, and Bargmann published the explicit
integral formula in 1961~\cite{HallHolomorphic2000}.
It maps square-integrable functions on $\mathbb{R}^n$ unitarily onto the
Bargmann-Fock space of entire functions on $\mathbb{C}^n$ with a Gaussian-weighted
norm~\cite{Bargmann}.
Its principal applications include coherent-state representations and quantization
in quantum mechanics for detailed analysis of annihilation and creation operators acting on a many-particle quantum state \cite{HallHolomorphic2000}, as well as time-frequency analysis through its relation to
the Gaussian-windowed short-time Fourier transform~\cite{Grochenig2001}.
Phase space is natural here because, up to scaling and sign conventions, the real
and imaginary coordinates label the position and momentum (or location and
frequency) of Gaussian wave packets, allowing the transform to describe where
oscillatory features occur as well as their frequencies~\cite{HallHolomorphic2000,Grochenig2001}.

We denote by $ \mathcal{F}_{a}^{2}(\C^{n}) $ the Bargmann-Fock space. We also denote by $ \mathcal{B}_{a} : L^{2}(\R^{n}; \C) \rightarrow \mathcal{F}_{a}^{2}(\C^{n}) $ and $ \mathcal{B}_{a}^{-1} :  \mathcal{F}_{a}^{2}(\C^{n}) \rightarrow L^{2}(\R^{n}; \C) $ the Segal-Bargmann transform and inverse-Segal-Bargmann transform, respectively. The precise definitions will be given in \S3. There is no difference if we restrict the inputs to real-valued
functions. In contrast with the current neural operator approaches, we would like to emphasize the advantages of our Segal-Bargmann neural operator approach:

(a) When $ f \in L^{2}(\R^{n}; \R) \cap C^{\infty}(\R^{n}; \R) $ and both {\it{Assumption A}} and {\it{Assumption B}} given in \S5 hold, Theorem \ref{REV2:thm2} says the pointwise difference between the estimate $ \hat{u} = \mathcal{G}_{NO}(\mathbf{f}) $ and the ground truth $ u $ satisfies the dimension-free convergence rate
\begin{equation*}
   \mathbb{E} \left[ \sup_{x \in \R^{n}} \lvert \hat{u} - u \rvert \right]= O\left(L^{-\frac{1}{2}} \right)
    \end{equation*}
with standard assumptions of the universal approximation theorem of neural operators on Hilbert spaces. Here $ L $ is the number of complex sampling points $ \{ z_{l} \}_{l = 1}^{L} \subset \C^{n} $, and is independent of the number of original sampling points $ \{ x_{i} \}_{i = 1}^{N} $. This canonical Monte-Carlo convergence rate is quite general in the sense that it is applicable to almost all partial differential equations, and has almost no requirement of the original grids. One of the advantages of this Segal-Bargmann NO approach is that we can select our complex grids as ``probes" or ``detectors" freely.

(b) {\it{Assumption A}} is too restrictive. If we drop the {\it{Assumption A}} and assume that $ f \in H^{s}(\R^{n}; \R) \cap \calC^{\infty}(\R^{n}; \R) $. If $ s > \frac{n}{2} $, then we give the dimension-free pointwise convergence rate
\begin{equation*}
\mathbb{E} \left[ \sup_{x \in \R^{n}} \lvert \hat{u} - u \rvert \right]= O\left(L^{-\frac{1}{2}} \right)
\end{equation*}
in Theorem \ref{REV2:thm4}, provided that {\it{Assumption B}} holds. If $ s < \frac{n}{2} $, then a natural consequence of Theorem \ref{REV2:thm3} and universal approximation theorem of neural operators will imply that
\begin{equation*}
    \mathbb{E} \left[ \lVert u - \hat{u} \rVert_{L^{2}(\R^{n}; \R)} \right]= O\left(L^{-\frac{s}{2n}} \right).
\end{equation*}
We give a theorem in the Appendix, which demonstrates that our {\it{Assumption B}} can be achieved with some hypotheses and is therefore a reasonable assumption in general.

(c) The $ L^{-\frac{1}{2}} $ reconstruction errors are dimension-free under the stated hypotheses, which indicates that our Segal-Bargmann neural operator approach partially overcomes the curse of dimensionality. In principle, we can approach to the true solution of the PDE by increasing the complex grids $ L \rightarrow \infty $, with just fixed number of observations on $ \{ x_{i} \}_{i =1}^{N} $. We give numerical results for the highly nonlinear PDE $ -\Delta u + e^{u} + u^{3} = f \; {\rm in} \; \Omega, u = 0 \; {\rm on} \; \partial \Omega $ for $ \Omega \subset \R^{5} $. Based on our best understanding, the state-of-the-art NO approaches for unknown differential operator $ \mathcal{L} $ is spatial dimension $ 3 $.

(d) The Segal-Bargmann neural operator approach integrates the information recovery and data encoding, the latent space representation, the learning process and the data decoding into a unified framework, summarized by the following diagram:
\begin{equation*}
\begin{tikzcd}[ column sep=8em, row sep=5em]
&u \in L^{2} \arrow[d, leftrightarrow, "\approx"] \\    
L^{2} \ni f \arrow[r, "\mathcal{G}_{\theta}", ] \arrow[ur, ur, "\mathcal{L}^{-1}"] \arrow[d, "\mathcal{B}_{a} + \text{Encoder}"] & \hat{u} \in L^{2} \\
\mathcal{F}_{a}^{2} \ni \mathcal{B}_{a}(f) \arrow[r, "\text{Latent Space Rep'n} + \hat{\mathcal{G}}_{\hat{\theta}}"] &  \mathcal{B}_{a}(\hat{u}) \in \mathcal{F}_{a}^{2} \arrow[u, "\mathcal{B}_{a}^{-1} + \text{Decoder}"].
\end{tikzcd}
\end{equation*}

(e) Our algorithm is fast and accurate since our learning process is undertaken in the Bargmann-Fock space, which is an RKHS consisting of complex analytic functions. The RKHS property and the rigidity of holomorphic functions allow us to use small neural network to achieve high accuracy with low computational cost. Our experiments report not only small $ L^{2} $-type errors, but also much smaller $ \mathbb{E}_{X}\left[ \frac{\lvert \hat{u}(X) - u(X)\rvert}{\lvert u(X) \rvert} \right] $ and $ \sup_{x} \frac{\lvert \hat{u}(x) - u(x) \rvert}{\lvert u(x) \rvert} $ errors than the NO baselines considered, i.e.\ we also measure the typical (expected) difference between the estimated solution $ \hat{u} $ and the ground truth $ u $. Our algorithm uses the Gaussian-windowed Fourier-transform nature of the Segal-Bargmann transform, which ensures the stability and high accuracy of our algorithm.

(f) Both the Segal-Bargmann and inverse-Segal-Bargmann transforms are unitary transformations. The key feature of this approach is that we decouple the observations of $ f $ in the original space at $ \{ x_{i} \} \in \R^{n} $ from the learning process with encoding and representation in the new space via new grids $ \{ z_{l} \} \subset \C^{n} $. Precisely speaking, the Segal-Bargmann transform serves not only as a unitary information transformation, but also as an encoder, while the RKHS representation of the encoded data in Bargmann-Fock space is designed as the latent space representation of the original data. The inverse-Segal-Bargmann transform can be designed as an exact decoder at the continuous level.
\medskip

This article is organized as follows: In \S2, we give a comprehensive summary of the state-of-the-art data-driven PDE approaches, especially over the last decade. In \S3, we give the necessary definitions of the Segal-Bargmann transform, the Bargmann-Fock space, the RKHS kernel, and an associated orthonormal basis in the Bargmann-Fock space that will be used in later sections. In \S4, we give an example in \S4.1 of how the Segal-Bargmann transform can convert a partial differential equation with solutions in $ L^{2}(\R^{n}; \R) $ into a new, complex-valued partial differential equation in $ \mathcal{F}_{a}^{2}(\C^{n}) $, where two solutions correspond through the Segal-Bargmann transform. We give the problem setting in \S4.2. In \S4.3, we give a two-step algorithm where we study the encoder $ W $ first, and then study the neural operator in the Bargmann-Fock space. In \S4.4, we give an end-to-end algorithm that combines the study of the encoder and the neural operator together. In \S5, we give theoretical results following the problem setting in \S4. With {\it{Assumption A}} and {\it{Assumption B}} , we show in Theorem \ref{REV2:thm2} that the error bound between the estimate $ \hat{u} $ and the solution $ u $ is dimension-free in the sense of the sup-norm, which is at least at the level of classical Monte-Carlo rate $ O\left( L^{-\frac{1}{2}} \right) $ where $ L $ is the number of the complex sampling points we choose, decoupled from the original real grids. Once we drop the {\it{Assumption A}} and consider the most general input $ f $ and solution $ u $, we show in Theorem \ref{REV2:thm4} that the same conclusion holds by assuming that $ f $ and $ u $ belong to suitable Sobolev spaces. In \S6, we give numerical results for 2D Darcy flow, 1D Burgers' equation 
with appropriate initial and boundary conditions. For both equations above, we also provide benchmark comparisons with established methods such as FNO, GNOT, Belnet, etc.\ for both uniform grids and scattered points. We also give numerical results for two nonlinear PDEs in spatial dimension $ 5 $. 
\medskip

\section{Related Work in Neural PDEs}
Our Segal-Bargmann neural operator approach is a neural operator that uniquely integrates the Segal-Bargmann encoding, operator learning in a finite-dimensional Bargmann-Fock RKHS space, and inverse Segal-Bargmann decoding into a unified framework for learning PDE solution operators. 

Over the past decade, learning-based PDE solvers have been significantly developed along two main directions:
approximating the solution of an individual problem and learning the operator that maps
coefficients, forcing terms, or initial and boundary conditions to solutions.
These approaches differ in their representations, physical constraints, and use of
numerical algorithms.

\subsection{Physics-informed neural networks (PINNs)}
PINNs represent the solution with a neural network and combine equation residuals,
initial and boundary conditions, and observations in a common loss, supporting both
forward simulation and parameter identification~\cite{PINN}.
Variational and weak formulations use integral constraints to reduce reliance on
pointwise high-order derivatives, while domain decomposition assigns local networks
to different regions~\cite{VPINN,WAN,XPINN,FBPINN}.
Studies of gradient imbalance and neural tangent kernels explain why loss components
can converge at different rates, motivating adaptive weighting and improved optimization~\cite{GradientPathology,PINNNTK}.
B-PINNs introduce Bayesian inference to quantify uncertainty in solutions and inferred
parameters under noisy observations~\cite{BPINN}.

\subsection{DeepONet and reduced-order operator learning}
Operator learning trains a model across input functions so that new problem instances
can be evaluated without solving each one from scratch.
DeepONet combines a branch network encoding the input function with a trunk network
encoding output coordinates to predict a continuous solution~\cite{DeepONet}.
MIONet extends this construction to multiple input functions, while POD-DeepONet uses
a reduced output basis extracted from solution data~\cite{MIONet,PODDeepONet}.
PCA-Net compresses both input and output functions and learns a map between their
coefficients~\cite{PCANet}.
Their accuracy depends on whether these finite representations preserve the main
variations in the solution.

\subsection{Spectral neural operators}
Spectral architectures provide global interactions through transforms.
The Fourier neural operator (FNO), a widely used architecture, learns Fourier multipliers
and combines them with local mappings and nonlinear activations~\cite{FNO}.
F-FNO factorizes spectral operations to reduce computational cost, whereas multiwavelet
operators and WNO use spatial localization and scale decomposition to represent
multiscale structure~\cite{FFNO,MWT,WNO}.
The Laplace neural operator (LNO) represents transient responses through poles and
residues in the Laplace domain~\cite{LNO}.
The choice of transform reflects the spatial and temporal structure of the solution.

\subsection{Graph networks, geometry, and attention}
Irregular meshes and varying domains require models that account for spatial geometry.
MeshGraphNets learns local interactions and time evolution through mesh-based message
passing; Geo-FNO and GINO connect physical domains to latent representations through
learned deformations or graph operators~\cite{MeshGraphNets,GeoFNO,GINO}.
Attention offers input-dependent interactions: Galerkin Transformer relates attention
to integral approximation, while Transolver aggregates physical states to reduce the
cost of interactions over large point sets~\cite{GalerkinTransformer,Transolver}.
These methods support solution prediction beyond regular grids, including complex
engineering geometries.

\subsection{Physics-informed neural operators (PINO)}
Physics-informed operator learning imposes governing equations across a family of inputs
to reduce dependence on reference solutions.
Early physics-constrained convolutional surrogates already explored parametric PDE
learning without paired solution data~\cite{PhysicsConstrainedCNN}.
PI-DeepONet incorporates equation and boundary residuals into DeepONet, while PINO
combines data fitting with physics constraints, allowing coarse-grid solution data and
fine-grid residual evaluation~\cite{PIDeepONet,PINO}.
Variational DeepONet and VINO instead formulate operator training through energy or
weak-form objectives~\cite{VariationalDeepONet,VINO}.
This approach combines rapid operator evaluation with known physical laws, although
residual discretization and loss balancing remain consequential training choices.

\subsection{Kernel methods and random features}
Kernel methods provide alternative representations for both PDE solutions and solution
operators. Gaussian process methods impose PDE constraints in a kernel space, while
operator kernel regression learns solution maps from finite observations of input and
output functions~\cite{GPNonlinearPDE,KernelCompetitive}.
Random features reduce the representation to a finite coefficient problem:
operator random-feature models learn maps between function spaces, and the Random
Feature Method combines local features with domain decomposition for PDE solving~\cite{RFOperator,RFMPDE}.
Fitting expansion coefficients simplifies optimization, although matrix size,
conditioning, and nonlinear iterations can still be costly.

\subsection{Learning within numerical solvers}
Hybrid methods retain numerical algorithms and learn components that are expensive
or difficult to design.
Data-driven discretizations estimate differential operators on coarse grids;
Solver-in-the-Loop includes numerical evolution during training; and HINTS combines
neural operators with classical relaxation to reduce complementary error components~\cite{DataDrivenDiscretization,SolverInLoop,HINTS}.
For high-dimensional PDEs, deep backward schemes and deep splitting use stochastic
representations and time discretization to obtain sequences of neural regression
problems~\cite{DBDP,DeepSplitting}.
The network thus learns a specific numerical subproblem within an established
solution procedure.

\subsection{Generative models for PDEs}
Generative methods model distributions of solution fields, supporting inference from
partial observations and uncertain data.
DiffusionPDE combines diffusion priors with observation and physics constraints, while
physics-informed diffusion models incorporate governing equations directly into
training~\cite{DiffusionPDE,PIDiffusion}.
PDE-Refiner uses iterative denoising to improve long autoregressive rollouts~\cite{PDERefiner}.
These approaches extend prediction to conditional sampling and iterative refinement,
with additional computational costs from repeated generation steps.

\subsection{PDE pretraining and foundation models}
Recent work pretrains shared models on multiple equations and physical datasets before
adapting them to specific tasks.
DPOT and Poseidon study representation learning and transfer across PDEs, while the
PDEformer family also encodes equations through symbolic or computational graph
representations~\cite{DPOT,Poseidon,PDEformer2}.
ICON uses input--output function pairs as context, enabling adaptation to new operator
tasks through examples~\cite{ICON}.
These models seek representations that can be reused across equations, conditions,
and data distributions.
\medskip

\section{The Segal-Bargmann Transform}
In this section, we introduce the Segal-Bargmann transform that will be used later. For more details we refer to \cite[Chapter 1]{Folland}. Throughout this article, our domain is either $ \R^{n} $ or open subsets of $ \R^{n} $. We also assume the familiarity of the standard $ L^{2} $-type Sobolev spaces $ H^{s}(\R^{n}; \R) $. We will repeatedly use the standard Gamma function $ \Gamma(s) $, the incomplete upper and lower Gamma functions $ \gamma(s, x) $ and $ \Gamma(s, x) $, respectively. 

Precisely speaking, the Segal-Bargmann transform is a Guassian-weighted Fourier transform with complex frequency. Throughout this article, we will interchangeably use Segal-Bargmann transform and $ a $-Segal-Bargmann transform. The latter one is used when we want to emphasize the importance of the parameter $ a $. All definitions below work equally well for complex-valued $ L^{2} $-functions $ f \in L^{2}(\R^{n}; \C) $. 
\begin{definition}\label{SB:de1}
Let $ L^{2}(\R^{n}; \R) $ be the collection of complex-valued square integrable measurable functions. We denote by the $ a $-Segal-Bargmann transform:
\begin{equation}\label{SB:eqn1}
    L^{2}(\R^{n}; \R) \ni f \mapsto \mathcal{B}_{a}(f)(z) : = \left( \frac{2a}{\pi} \right)^{\frac{n}{4}} \int_{\mathbb{R}^n} \exp \left( -\frac{a}{2} z \cdot z - a x \cdot x + 2a z \cdot x \right) f(x) dx.
\end{equation}
The function
\begin{equation*}
    A_{a}(z, x) = \left( \frac{2a}{\pi} \right)^{\frac{n}{4}} \exp \left( -\frac{a}{2} z \cdot z - a x \cdot x + 2a z \cdot x \right)
\end{equation*}
is called $ a $-Segal-Bargmann kernel.
\end{definition}
It is constructed by Segal and Bargmann that the function $ B_{a}(f) $ is an entire function on $ \C^{n} $. Furthermore, it is sitting in a Hilbert space, called Bargmann-Fock space.
\begin{definition}\label{SB:def2}
The Bargmann-Fock space $ \mathcal{F}_{a}^{2}(\C^{n}) $ is defined to be the collection of entire functions such that
\begin{equation*}
\mathcal{F}_{a}^{2}(\mathbb{C}^{n}) : = \lbrace F \in \mathcal{O}(\mathbb{C}^{n}) : \lVert f \rVert_{\mathcal{F}_{a}^{2}}^{2} = \langle F, F \rangle_{\mathcal{F}_{a}^{2}(\mathbb{C}^{n})} < \infty \rbrace
\end{equation*}
where the Bargmann-Fock inner product is defined as
\begin{equation*}
    \langle F, G \rangle_{\mathcal{F}_{a}^{2}(\mathbb{C}^{n})} : = \int_{\mathbb{C}^{n}} F(z) \overline{G(z)} d\mu_{a}(z) : =  \left( \frac{a}{\pi} \right)^{n} \int_{\mathbb{C}^{n}} F(z) \overline{G(z)} e^{-a \lvert z \rvert^{2}} d\mu(z).
\end{equation*}
Here $ d\mu_{a}(z) = \left( \frac{a}{\pi} \right)^{n} e^{-a \lvert z \rvert^{2}} d\mu(z) $ is the $ a $-Gaussian-weighted complex measure, and $ d\mu(z) $ is the standard complex Lebesgue measure. When $ n = 1 $, $ d\mu(z) = dxdy $ for $ z = x + \sqrt{-1} y $.
\end{definition}
The Segal-Bargmann transform is a unitary transform in the sense that
\begin{equation*}
    F = \mathcal{B}_{a}(f), G = \mathcal{B}_{a}(g) \Rightarrow \lVert f \rVert_{L^{2}(\R^{n}; \R)} = \lVert F \rVert_{\mathcal{F}_{a}^{2}(\C^{n})}, \langle F, G \rangle_{\mathcal{F}_{a}^{2}(\C^{n})} = \langle f, g \rangle_{L^{2}(\R^{n}; \R)}.
\end{equation*}
Therefore, the $ a $-Segal-Bargmann transform is an isomorphism between $ L^{2}(\R^{n}; \R) $ and $ \mathcal{F}_{a}^{2}(\C^{n}) $. The inverse $ a $-Segal-Bargmann transform is defined below.
\begin{definition}\label{SB:def3}
The inverse $ a $-Segal-Bargmann transform is
\begin{equation*}
    \mathcal{B}_{a}^{-1}(F)(x) : = \left( \frac{2a}{\pi} \right)^{\frac{n}{4}} \cdot \left( \frac{a}{\pi} \right)^{n} \int_{\mathbb{C}^n} e^{-\frac{a}{2} \bar{z} \cdot \bar{z} - a x \cdot x + 2a \bar{z} \cdot x} F(z) e^{-a\lvert z \rvert^{2}} d\mu(z)
\end{equation*}
such that
\begin{equation*}
    \mathcal{B}_{a}(\mathcal{B}_{a}^{-1}) = Id_{\mathcal{F}_{a}^{2}(\C^{n})}, \mathcal{B}_{a}^{-1}(\mathcal{B}_{a}) = Id_{L^{2}(\R^{n}; \R)}.
\end{equation*}
\end{definition}
It is well-known that the Bargmann-Fock space is an RKHS space, i.e. the point evaluation functional is a continuous functional on $ \mathcal{F}_{a}^{2}(\C^{n}) $. We now state a few more facts about the space $ \mathcal{F}_{a}^{2}(\C^{n}) $.
\begin{definition}\label{Def4}
There exists an orthonormal basis $ \{ e_{\beta}^{a} \}_{\beta} $ of $\mathcal{F}_{a}^{2}(\C^{n})$ with
\begin{equation*}
    e_{\beta}^{a}(z) = \sqrt{\frac{a^{\lvert \beta \rvert}}{\beta !}} z^{\beta}, \beta \in \mathbb{N}_{0}^{n}.
\end{equation*}
This orthonormal basis associates with an orthonormal basis $ \{ h_{\beta}^{a} \}_{\beta} $ in $ L^{2}(\R^{n}; \R) $ due to the isometry of $ a $-Segal-Bargmann transform. The normalized Gaussian-Hermite function $ h_{\beta}^{a} $ is of the form
\begin{equation*}
    \mathcal{B}_{a}^{-1}(e_{\beta}^{a})(x) = h_{\beta}^{a}(x) = \frac{(2a)^{\frac{n}{4}}}{\pi^{\frac{n}{4}}\sqrt{2^{\lvert \beta \rvert}\beta !}} H_{\beta}(\sqrt{2a}x) e^{-\frac{a\lvert x \rvert^{2}}{2}}.
\end{equation*}
Here $ \beta = (\beta_{1}, \dotso, \beta_{n}) $, 
\begin{equation*}
    H_{\beta}(x) = \Pi_{i =1}^{n} H_{\beta_{i}}(x_{i}), x = (x_{1}, \dotso, x_{n}) \in \R^{n}
\end{equation*}
where the $ 1 $-dimensional physicists' Hermitian polynomial of order $ m $ is given by
\begin{equation*}
    H_{m}(y) = (-1)^{m} e^{y^{2}} \frac{d}{dy^{n}} e^{-y^{2}}, y \in \R.
\end{equation*}
The RKHS kernel of the $ a $-Bargmann-Fock space is
\begin{equation*}
    K_{a}(z, w) = e^{a z \cdot \bar{w}} = \sum_{\beta} e_{\beta}^{a}(z) e_{\beta}^{a}(\bar{w}).
\end{equation*}
such that
\begin{equation*}
\langle F(\cdot), K_{a}(\cdot, w) \rangle_{\mathcal{F}_{a}^{2}(\C^{n})} = F(w), \forall F \in \mathcal{F}_{a}^{2}(\C^{n}).
\end{equation*}
\end{definition}
\begin{definition}\label{Def5}
Define the Gauss-Hermite harmonic oscillator
\begin{equation*}
   \mathcal{H}_{\beta}^{a} : = -\Delta + 4a^{2} \lvert x \rvert^{2}, 
\end{equation*}
which satisfies the following eigenvalue equation
\begin{equation*}
   \mathcal{H}_{a} h_{\beta}^{a} = (4a \lvert \beta \rvert + 2a n) h_{\beta}^{a},
\end{equation*}
i.e.\ each $ h_{\beta}^{a} $ is an eigenfunction of the operator $ \mathcal{H}_{a} $ with eigenvalue $ \lambda_{\beta, n} = (4a \lvert \beta \rvert + 2a n) $. 

Define the even order Hermite-Sobolev space
\begin{equation*}
    \mathcal{H}_{a}^{s}(\R^{n}; \R) : = \{ f \in L^{2}(\R^{n}; \R) : \lVert f \rVert_{\mathcal{H}_{a}^{s}(\R^{n}; \R} : = \lVert (I + \mathcal{H}_{a})^{\frac{s}{2}} f \rVert_{L^{2}(\R^{n}; \R)} < \infty, s = 2m \in \mathbb{N} \}.
\end{equation*}
It follows that in the $ L^{2} $-sense
\begin{equation*}
    f = \sum_{\beta} c_{\beta} h_{\beta}^{a} \Rightarrow \lVert f \rVert_{\mathcal{H}_{a}^{s}(\R^{n}; \R)}^{2} = \sum_{\beta}( 1 + \lambda_{\beta, n})^{s}\lvert c_{\beta} \rvert^{2} = \sum_{\beta}(1 + 4a \lvert \beta \rvert + 2an)^{s} \lvert c_{\beta} \rvert^{2}.
\end{equation*}
\end{definition}

\section{Neural PDEs and Algorithms}
In this section, we first give a concrete example of how the Segal-Bargmann transform converts a PDE where the differential operator $ \mathcal{L} $ applies to the orignial $ L^{2}(\R^{n}; \R) $ to a new PDE where the differential operator $ \tilde{G} $ applies to the Bargmann-Fock space $ \mathcal{F}_{a}^{2}(\C^{n}) $. This concrete example indicates that study the original PDE is equivalent to study the new PDE after Segal-Bargmann transform.

We then set the problem and give two algorithms of how to learn the differential operator, or more interestingly, the inverse of the differential operator, by a neural operator that maps between Bargmann-Fock spaces. The theoretical analysis of our approaches, including the convergence rates, are established in \S5 with moderate assumptions.
\medskip

\subsection{Neural PDEs in Bargmann-Fock Space}
We would like to apply the Segal-Bargmann transform to solve partial differential equations with appropriate initial and boundary conditions. As an example, we consider the following elliptic equation
\begin{equation*}
    -\Delta u + u = f \; {\rm in} \; \R^{n}
\end{equation*}
where $ f \in L^{2}(\R^{n}; \R) $. In our scenario, we assume that we have some prior knowledge of $ u $ and $ f $, and we would like to use Neural Operator (NO) to recover the inverse $ (-\Delta + Id)^{-1} $. Instead of the classical methods that construct the NO between the spaces of functions over $ \R^{n} $, we take the Segal-Bargmann transform and consider the differential relation in $ \mathcal{F}_{a}^{2}(\C^{n}) $, where we take advantage of the fact that $ \mathcal{F}_{a}^{2}(\C^{n}) $ is an RKHS space with a designated RKHS kernel.

Let $ (z_{1}, \dotso, z_{n}) $ be coordinates in $ \C^{n} $. Since $ f \in  \calL^{2}(\R^{2}; \R) $, we are legitimate to apply the $ a $-Segal-Bargmann transform on both sides of the PDE, 
\begin{equation*}
    \mathcal{B}_{a}(-\Delta u + u)(z) = \mathcal{B}_{a}(f)(z) \in \mathcal{F}^{2}(\C^{n}).
\end{equation*}
Applying integration by parts on $ \R^{n} $, the $ a $-Segal-Bargmann transform for both sides of the PDE gives
\begin{equation}\label{AL:eqn1}
    (L_{z, n} + 1) \mathcal{B}_{a}(u) : = -\sum_{j = 1}^{n} \left( \frac{\partial^{2}\mathcal{B}(u)}{\partial z_{j}^{2}} + \left(-a + a^{2} z_{j}^{2} - 2z_{j} a \partial_{z_{j}} \right)\mathcal{B}_{a}(u)\right) + \mathcal{B}_{a}(u) = \mathcal{B}_{a}(f).
\end{equation}
Assume that we have inputs and outputs $ \{ f_{k} \} $, $ \{ u_{k} \} $, respectively. Our goal is to learn the solver, i.e. we are seeking for a NO $ \mathcal{G}_{\theta} $ parametrized by $ \theta \in \R^{n} $, and we hope that $ \mathcal{G}_{\theta} f $ solves the PDE, i.e.
\begin{equation*}
    (-\Delta + 1) (\mathcal{G}_{\theta} f) \approx f.
\end{equation*}
In other words, we would like to study $ \mathcal{G}_{\theta} \approx (-\Delta + 1)^{-1} $. With the help of the $ a $-Segal-Bargmann transform, our study diagram can be shown as below:
\begin{equation}\label{AL:eqn1a}
\begin{tikzcd}
    \{ f_{k} \} \subset L^{2}(\R^{n}; \R) \arrow[r, "\mathcal{G}_{\theta}"] \arrow[d, "\mathcal{B}_{a}"] & H^{2}(\R^{n}; \R) \supset \{ u_{k} \} \\
   \{ \mathcal{B}_{a}(f_{k}) \} \subset \mathcal{F}_{a}^{2}(\C^{n}) \arrow[r, "\hat{\mathcal{G}}_{\hat{\theta}}"] & \mathcal{F}_{a}^{2}(\C^{n}) \supset \{ \mathcal{B}_{a}(u_{k}) \} \arrow[u, "\mathcal{B}_{a}^{-1}"].
\end{tikzcd}
\end{equation}
Due to the one-to-one unitary correspondence between $ \mathcal{L}^{2}(\R^{2}) $ and $ \mathcal{F}^{2}(\C^{2}) $, it is equivalent to learn the complexified Neural Operator $ \hat{\mathcal{G}}_{\hat{\theta}} $ instead such that
\begin{equation*}
    (L_{z, n} + 1)\hat{\mathcal{G}}_{\hat{\theta}} \approx Id_{\mathcal{F}_{a}^{2}(\C^{n})}.
\end{equation*}
\medskip
\subsection{The Problem Setting and The Loss of Information of Discrete Segal-Bargmann Transform} Our problem setting is quite general.

Consider the time-independent PDEs
\begin{equation*}
    \mathcal{L}u = f \; {\rm in} \; \Omega \subset \R^{n}, Bu = g \; {\rm on} \; \partial \Omega.
\end{equation*}
Here $ u, f \in L^{2}(\R^{n}; \R), g \in L^{2}(\R^{n - 1}; \R) $, $ \mathcal{L} $ is a smooth differential operator, and $ B $ is a smooth boundary operator. $ \Omega $ is a bounded, open, connected subset of $ \R^{n} $ for simplicity, and $ \partial \Omega $ may have some level of singularity, e.g.\ non-Lipschitz, etc.

We also consider the time-dependent PDEs
\begin{equation*}
    \mathcal{L}_{t} u + \mathcal{L}_{x} u = f \; {\rm in} \; \Omega \times [0, T] \subset \R^{n} \times \R_{t}, B_{x} u = g \; {\rm on} \; \partial \Omega \times [0, T], B_{t}u = u_{0} \; {\rm on} \; \Omega \times \{ 0 \}. 
\end{equation*}
Here $ u, f \in \calC([0, T]; L^{2}(\R^{n}; \R)) $, $ g \in \calC([0, T]; L^{2}(\R^{n - 1}; \R))$, $ u_{9} \in L^{2}(\R^{n}; \R) $, $ \mathcal{L}_{t}, \mathcal{L}_{x}, B_{x}, B_{t} $ are smooth differential and boundary operators, respectively.

Under this basic settings, it is legitimate to take $ a $-Segal-Bargmann transform on both sides of the PDE, which will result in either a local version as the example shows, or at least a superposition.

In this article, we apply machine learning techniques to estimate the inverse of the differential operators $ \mathcal{L}^{-1} $ or $ (\mathcal{L}_{x} + \mathcal{L}_{t})^{-1} $ by neural operators $ \mathcal{G}_{NO} $, where the input is $ f $, and the output is $ u $, provided that the boundary and initial conditions are completely known.

In particular, we only have knowledge of a collection of training set $ (u^{(k)}, f^{(k)})_{k = 1}^{K} $ where we do not know the functions but instead a finite collection of observations $ (u^{(k)}(x_{i}), f^{(k)}(x_{i}))_{i = 1}^{N} $ for each $ k $ only.

When the spatial dimension $ n $ is high, or when the sampling size $ N $ or $ K $ is relatively small, or we have no prior knowledge of how the observations are collected, it is very hard to recover the relation between $ u $ and $ f $ efficiently just within the original space $ \R^{n} $ and associated Banach spaces.

For example, $ f = \sum_{\beta} c_{\beta} h_{\beta}^{a} \in L^{2}(\R^{n}; \R) $ uniquely. An encoder is necessary to recover $ \{ c_{\beta} \} $ from $ \{ f(x_{i}) \} $. However, there is no theoretical guarantee of this step. Nevertheless, the best we can do is to recover $ \hat{c}_{\beta} $ such that $ f(x_{i}) \approx \sum_{\beta} \hat{c}_{\beta} h_{\beta}^{a}(x_{i}) $. This indicates the significant demand of observations if we would like to recover $ f $ everywhere. 

However, we may only have very limited observations over $ \{ x_{i} \}_{i = 1}^{N} $. One of the main purposes of the introduction of the Segal-Bargmann transform and the Bargmann-Fock space is to recover the information of $ f $ by complex ``probes" $ \{ z_{l} \} $, which can be chosen arbitrarily (assume that we can afford the cost).

Not only that, we can also give a concrete latent space representation and then decode the data by taking advantage of the RKHS structure of the Bargmann-Fock space. In addition, the elements of $ \mathcal{F}_{a}^{2}(\C^{n}) $ are entire functions, and the rigidity of complex analyticity ensures the uniqueness of our decoder theoretically.

As we said, the purpose of the Segal-Bargmann encoder is due to the lack of information. For the ground truth $ f \in L^{2}(\R^{n}; \R) $, the unitary transformation that preserves all information of $ f $ is
\begin{equation*}
    \mathcal{B}_{a}(f)(z) = \int_{\R^{n}} A_{a}(z, x) f(x) dx.
\end{equation*}
But in practice, we often only have finitely many observations $ \{ f(x_{i}) \}_{i = 1}^{N} $; in addition, we often have no idea whether the real grids $ \{ x_{i} \}_{i = 1}^{N} $ was deliberately chosen with some pattern or not. 

Clearly, the naive discrete Segal-Bargmann transform
\begin{equation*}
    f \mapsto \frac{1}{N} \sum_{i=1}^{N} A_{a}(z, x_{i}) f(x_{i})
\end{equation*}
is not a good choice. Denote by $ F = \mathcal{B}_{a}(f) $ and suppose that we have chosen a complex sampling $ \{ z_{l} \}_{l = 1}^{N} $, our hope is that we can find good approximations $ \tilde{F}_{N} $ such that
\begin{equation*}
    F(z_{l}) \approx \tilde{F}_{N}(z_{l}), \forall l.
\end{equation*}
If we can recover $ F $, it is equivalent to recover $ f $. We therefore introduce an {\it{information recovery matrix}} $ W = (w_{li}) $ and denote
\begin{equation*}
    \tilde{F}_{N}(z_{l}) : = \frac{1}{N} \sum_{i =1}^{N} w_{li} A_{a}(z_{l}, x_{i})f(x_{i}).
\end{equation*}
Note that the product of $ W $ and $ A = (A_{a}(z_{l}, x_{i})) $ is a Hadamard product. Theoretically, we can apply a nonlinear activation on $ \tilde{\mathbf{F}}_{N} = (\tilde{F}_{N}(z_{1}), \dotso, \tilde{F}_{N}(z_{L}))^{t} $ to recover $ F(z_{l}) $ with zero error when $ L $ is large enough by a version of universal approximation theorem \cite{DU}.

We then use our $ \tilde{F}_{N}(z_{l}) $ to construct the encoder. Intuitively, we improve our learning process in (\ref{AL:eqn1}) via Segal-Bargmann transform as the following diagram shows
\begin{equation}\label{AL:eqn1c}
\begin{tikzcd}
    \{ f_{k} \} \subset L^{2} \arrow[r, "\mathcal{G}_{\theta}"] \arrow[d, "\mathcal{B}_{a}"] & L^{2} \supset \{ u_{k} \} \\
   \{ \mathcal{B}_{a}(f_{k}) \} \subset \mathcal{F}_{a}^{2} \arrow[r, "\hat{\mathcal{G}}_{\hat{\theta}}"] & \mathcal{F}_{a}^{2} \supset \{ \mathcal{B}_{a}(u_{k}) \} \arrow[u, "\mathcal{B}_{a}^{-1}"].
\end{tikzcd} \Rightarrow \begin{tikzcd}
    \{ f_{k} \} \subset L^{2} \arrow[r, "\mathcal{G}_{\theta}"] \arrow[d, "\mathcal{B}_{a} + W "] & L^{2} \supset \{ u_{k} \} \\
   \{ \mathcal{B}_{a}(f_{k}) \} \subset \mathcal{F}_{a}^{2} \arrow[r, "\hat{\mathcal{G}}_{\hat{\theta}}"] & \mathcal{F}_{a}^{2} \supset \{ \mathcal{B}_{a}(u_{k}) \} \arrow[u, "\mathcal{B}_{a}^{-1}"].
\end{tikzcd}
\end{equation}
Therefore, we need to learn both $ W $ and $ \hat{\mathcal{G}}_{\hat{\theta}} $. We give two algorithms below.
\medskip

\subsection{Two-Step Algorithm} 
In this subsection, we give a twp-step algorithm. The first step is to construct a Segal-Bargmann encoder as shown in (\ref{AL:eqn1c}); the second step is to learn the complex version of the differential operator over the Bargmann-Fock space.

{\it{Step 0: Original Input.}} Assume that we have training set $ \{ u^{(k)}, f^{(k)} \} $ with $\mathcal{L} u^{(k)} = f^{(k)} $. Suppose that we have real grids $ \{ x_{i} \}_{i = 1}^{N} $, either from some distribution $ \{ X_{i} \} \overset{i.i.d.}{\sim} \rho $ or not, and we only have evaluation representations
\begin{equation*}
    \{ u^{(k)}(x_{i}), f^{(k)}(x_{i}) \}, i = 1, \dotso, N, k = 1, \dotso, K.
\end{equation*}
\medskip

{\it{Step 1: Complex Grids and Segal-Bargmann Encoder.}} Denote by
\begin{equation*}
    \Lambda_{M} = \lbrace \beta \in \mathbb{N}_{0}^{n} : \lvert \Lambda_{M} \rvert = M \rbrace, A_{a, M}(z, x) = \sum_{\lvert \beta \rvert \leqslant M} e_{\beta}^{a}(z) h_{\beta}^{a}(x),
\end{equation*}
We denote by $ V_{M} = \{ h_{\beta}^{a} : \beta \in \Lambda_{M} \} $ the associated finite collection of orthonormal basis with respect to $ \Lambda_{M} $, and $ B_{M} = \sup_{\beta \in \Lambda_{M}} \lvert \beta \rvert $. Set
\begin{equation*}
    K_{M}(z) : = \sum_{\beta \in \Lambda_{M}} \lvert e_{\beta}^{a}(z) \rvert^{2}, q_{M}(z) = \frac{M}{K_{M}(z)},  d\rho_{a}(z): = \frac{K_{M}(z)}{M} d\mu_{a}(z) = \frac{K_{M}(z)}{M} \left( \frac{a}{\pi} \right)^{n} e^{-a \lvert z \rvert^{2}} d\mu(z).
\end{equation*}
Applying Christoffel-function sampling following Cohen-Migliorati's \cite{CM}, we choose complex grids as random sampling of $ \{ Z_{l} \}_{l = 1}^{\frac{L}{2}} $ for even $ L $  such that
\begin{equation*}
    Z_{1}, \dotso, Z_{L/2} \overset{i.i.d.}{\sim} \rho_{a}.
\end{equation*}
Our complex sampling of even size $ L $ is given by complex conjugacy
\begin{equation*}
Z_{1}, \dotso, Z_{L/2}, \bar{Z}_{1}, \dotso, \bar{Z}_{L/2} = Z_{1}, \dotso, Z_{L/2}, Z_{L/2 + 1}, \dotso, Z_{L}.
\end{equation*}
Let $ \{ z_{l} \}_{l = 1}^{L} $ be the random sampling from the collection of random variables above.  We define
\begin{equation*}
    \tilde{F}_{N}^{k}(z_{l}) : = \frac{1}{N} \sum_{i = 1}^{N} w_{li} A_{a, M}(z_{l}, x_{i})f^{(k)}(x_{i}).
\end{equation*}
Here $ W = (w_{li}) $ is the {\it{information recovery matrix}}. For general $ f $ evaluated on $ \{ x_{i} \} $, we denote by
\begin{equation*}
    \mathbf{g}_{l} = \begin{pmatrix} A_{a, M}(z_{l}, x_{1})f(x_{1}) \\ \vdots \\ A_{a, M}(z_{l}, x_{N})f(x_{N}) \end{pmatrix}, \mathbf{w}_{l} = \begin{pmatrix} w_{l1} \\ \vdots \\ w_{lN} \end{pmatrix}, \mathbf{\mu}_{l} = \mathbb{E}_{f}\big[ \mathbf{g}_{l} \big], \Gamma_{lm} = \mathbb{E}_{f}\left[ (\mathbf{\mu}_{l} - \mathbf{g}_{l})(\mathbf{\mu}_{m} - \mathbf{g}_{m})^{*}\right]
\end{equation*}
where $ (\cdot )^{*} $ means Hermitian transpose. Theoretically, $ W $ is obtained by minimizing the loss
\begin{align*}
    L_{W, theory} & = \frac{\lambda_{1}}{L} \sum_{l = 1}^{L} \lvert \mathbf{\mu}_{l} \rvert^{2} + \frac{\lambda_{2}}{L(L - 1)} \sum_{l \neq m} \lvert \Gamma_{lm} \rvert^{2} + \frac{\lambda_{3}}{L} \sum_{l = 1}^{L} \lvert \Gamma_{ll} \rvert^{2} + \lambda_{4} \lVert W \rVert^{2}
\end{align*}
with regularization factors $ \lambda_{i}, i = 1, \dotso, 4$.Using our data $ \{ f^{(k)}(x_{i}) \} $ and $ \{ z_{l} \} $, we denote by
\begin{align*}
    \eta_{l}^{k} & = \frac{1}{N} \sum_{i = 1}^{N} w_{li} A_{a, M}(z_{l}, x_{i}) f^{(k)}(x_{i}) - f^{(k)}(z_{l}) \\
    \hat{\mathbf{\mu}}_{l} & = \frac{1}{K} \sum_{k = 1}^{K} \eta_{l}^{k}, \hat{\Gamma}_{lm} = \frac{1}{K} \sum_{k = 1}^{K} (\hat{\mathbf{u}}_{l} - \eta_{l}^{k})(\hat{\mathbf{u}}_{m} - q_{M}^{k})^{*},
\end{align*}
we get our practical loss as
\begin{align*}
    L_{W, data} & = \frac{\lambda_{1}}{L} \sum_{l = 1}^{L} \lvert \hat{\mathbf{\mu}}_{l} \rvert^{2} + \frac{\lambda_{2}}{L(L - 1)} \sum_{l \neq m} \lvert \hat{\Gamma}_{lm} \rvert^{2} + \frac{\lambda_{3}}{L} \sum_{l = 1}^{L} \lvert \hat{\Gamma}_{ll} \rvert^{2} + \lambda_{4} \lVert W \rVert^{2}.
\end{align*}
\medskip

{\it{Step 2: Complex Input.}} Denote by
\begin{equation*}
    \mathbf{e}^{a}(z) = \begin{pmatrix} e_{1}^{a}(z) & \dotso & e_{M}^{a}(z) \end{pmatrix}^{t},
\end{equation*}
we then define our complex inputs
\begin{equation}\label{AL:eqn0}
    \mathbf{c}_{0}^{k} : = \frac{1}{L} \sum_{l = 1}^{L} q_{M}(z_{l}) \overline{\mathbf{e}^{a}(z_{l})} \tilde{F}_{N}^{k}(z_{l}).
\end{equation}
Note that the holomorphy of $ e_{\beta}^{a} $, $ \tilde{F}_{N}^{k} $ and the complex conjugacy symmetry of $ \{ Z_{l} \} $ imply that $ \mathbf{c}^{k} $ are real-valued $ M $-vectors.
\medskip

{\it{Step 3: The Neural Operator.}} Note that $ \lvert \Lambda_{M} \rvert = M $, we want to set the latent space representation $ F_{N, 0}^{k} $ such that
\begin{equation*}
    B_{a}(f^{(k)}) = f^{(k)} \approx F_{N,0}^{k} = \sum_{\beta \in \Lambda_{M}} c_{\beta, 0}^{k} e_{\beta}^{a}.
\end{equation*}
Here the concrete latent space is given by $ \overline{\text{Span}{V_{M}}} $. It is equivalent to define the map from evaluation representations of $ f^{(k)} $ (and $ u^{(k)} $) to $ \mathbf{c}^{k} $ given in {\it{Step 2}}, which can also considered as a map from $ L^{2}(\R^{n};\R) $ to $ \mathcal{F}_{a}^{2}(\C^{n}) $,
\begin{align*}
    & \tilde{B} : \R^{N} \rightarrow \R^{M},  \begin{pmatrix} f^{(k)}(x_{1}) \\ \vdots \\ f^{(k)}(x_{N}) \end{pmatrix} = \mathbf{f}^{k} \mapsto \mathbf{c}_{0}^{k},  \tilde{B} : L^{2}(\R^{n};\R) \rightarrow \mathcal{F}_{a}^{2}(\C^{n}), f^{(k)} \mapsto F_{N,0}^{k}.
\end{align*}
With inverse Segal-Bargmann transform, we would hope without learning process (as we proved in Theorem \ref{CON:thml1} below)
\begin{align*}
    h_{\beta}^{a} = \mathcal{B}_{a}^{-1}(e_{\beta}^{a}) \Rightarrow & f^{(k)}(x) \approx \mathcal{B}_{a}^{-1}(F_{N,0}^{k})(x) = \sum_{\beta \in \Lambda_{M}} c_{\beta, 0}^{k}h_{\beta}^{a}(x) : = f_{N,0}^{k}(x).
\end{align*}
This induces a map from $ \R^{M} $ to $ \R^{N} $ by
\begin{align*}
   & \tilde{B}^{-1} : \R^{M} \rightarrow \R^{N}, \mathbf{c}_{0}^{k} \mapsto \mathbf{c}_{N, 0}^{k} : = \begin{pmatrix} f_{N, 0}^{k}(x_{1}) \\ \vdots \\ f_{N, 0}^{k}(x_{N}) \end{pmatrix}, \\
   & \tilde{B}^{-1} : \mathcal{F}_{a}^{2}(\C^{n}) \rightarrow L^{2}(\R^{n}; \R), F_{N, 0}^{k}(z) = \sum_{\beta \in \Lambda} c_{\beta, 0}^{k}e_{\beta}^{a}(z) \mapsto \sum_{\beta \in \Lambda_{M}} c_{\beta, 0}^{k}h_{\beta}^{a}(x) = f_{N, 0}^{k}(x).
\end{align*}
Relabeled our real input
\begin{equation*}
    v^{0,k} = f^{(k)}, \mathbf{c}_{0}^{k} = \tilde{B}(\mathbf{f}^{k}) = \tilde{B}(\mathbf{v}^{0, k}) = \mathbf{d}^{0,k}. 
\end{equation*}
For the $ 1st $ Segal-Bargmann layer, we apply a $ m $-layer MLP $ \mathcal{G}_{\theta_{1}} $ to $ \mathbf{c}_{0}^{k} $ such that
\begin{equation*}
   \R^{M} \ni \mathbf{\eta}^{1,k} = \mathcal{G}_{\theta_{1}}\mathbf{d}^{0,k} \Rightarrow \tilde{\mathbf{\eta}}^{1, k} = \tilde{B}^{-1}(\mathbf{\eta}^{1, k}).
\end{equation*}
The map $ \tilde{B}^{-1} $ is exactly the same as given above, which means
\begin{equation*}
\tilde{\mathbf{\eta}}^{1, k} = \tilde{B}^{-1}(\mathbf{\eta}^{1, k}) = \begin{pmatrix} \mathcal{B}_{a}^{-1}(\sum_{\beta \in \Lambda_{M}} \eta_{\beta}^{1, k} e_{\beta}^{a})(x_{1}) \\ \vdots \\ \mathcal{B}_{a}^{-1}(\sum_{\beta \in \Lambda_{M}} \eta_{\beta}^{1, k} e_{\beta}^{a})(x_{N}).
\end{pmatrix}
\end{equation*}
Adding a linear part and a nonlinear activation $ \sigma $ (taking GeLU or ReLU), we have the output of the first Segal-Bargmann layer, or the input of the second Segal-Bargmann layer
\begin{align*}
    \mathbf{v}^{1, k} & = \sigma(W_{0} \mathbf{v}^{0, k} + \tilde{\eta}^{1, k}) = \sigma(W_{0} \mathbf{v}^{0, k} + \tilde{B}^{-1} \mathcal{G}_{\theta_{1}}  \mathbf{d}^{0, k}) \\
    & = \sigma(W_{0} \mathbf{v}^{0, k} + \tilde{\eta}^{1, k}) = \sigma(W_{0} \mathbf{v}^{0, k} + \tilde{B}^{-1} \mathcal{G}_{\theta_{1}} \tilde{B} \mathbf{v}^{0, k}).
\end{align*}
Or in the sense of functions
\begin{equation*}
    v^{1, k}(x) = \sigma(W_{0}v^{0, k}(x) + \mathcal{B}^{-1}(\mathcal{G}_{\theta_{1}}\tilde{B}(v^{0, k})(Z_{l}))(x)).
\end{equation*}
Then we can do the second Segal-Bargmann layer, and get either complex output
\begin{equation*}
\mathbf{d}^{2, k} = \mathcal{G}_{\theta_{2}}\tilde{B}(\mathbf{v}^{1, k}).
\end{equation*}
or we can get another real output
\begin{equation*}
     \mathbf{v}^{2, k}  = \sigma(W_{1} \mathbf{v}^{1, k} + \tilde{B}^{-1} \mathcal{G}_{\theta_{1}} \tilde{B} \mathbf{v}^{1, k}).
\end{equation*}
or in the sense of function
\begin{equation*}
    v^{2, k}(x) = \sigma(W_{1}v^{1, k}(x) + \tilde{B}^{-1}(\mathcal{G}_{\theta_{2}}\tilde{B}(v^{1, k})(Z_{l})(x))
\end{equation*}
And we can keep going to get a total of $ S $ Segal-Bargmann layers, where the complex output and real output are denoted by
\begin{align}\label{AL:eqn5}
    \mathbf{d}^{S, k} & = \mathcal{G}_{\theta_{S}}\tilde{B}(\mathbf{v}^{S - 1, k}), \nonumber \\
    v^{S, k}(x) & = W_{S -1, k}v^{S - 1, k}(x) + \tilde{B}^{-1}(\mathcal{G}_{\theta_{S}}\tilde{B}(v^{S- 1, k})(Z_{l})(x).
\end{align}
\medskip

{\it{Step 4: The Real Relative $ L^{2} $-Loss.}}
By the expression of outputs in (\ref{AL:eqn5}), we define a relative $ L^{2} $-loss in the original real space. Set
\begin{equation*}
    \hat{u}^{k} : = v^{S, k} \Rightarrow \hat{u}^{k} : = \mathcal{G}_{\Theta}v^{0, k} = \mathcal{G}_{\Theta}f^{(k)},
\end{equation*}
we have
\begin{equation}\label{AL:eqn7}
L_{MSE, real} = \frac{1}{K} \sum_{k = 1}^{K} \frac{\sum_{i = 1}^{N} \lvert \hat{u}^{k}(x_{i}) - u^{(k)}(x_{i}) \rvert^{2}}{\sum_{i = 1}^{N} \lvert u^{(k)}(x_{i}) \rvert^{2}}.
\end{equation}
\medskip

{\it{Step 5: Testing.}}

Suppose that we have testing set $ \{ u_{test}^{k}, f_{test}^{k} \} $. Let's assume that our complex trained output is $ \hat{U}_{test}^{k} $ constructed by $ \hat{\mathbf{d}}_{test}^{k} = \tilde{\mathcal{G}}_{\Theta} \mathbf{d}_{test}^{0, k} $ where the input is $ \mathbf{d}_{test}^{0, k} = \mathbf{c}_{test}^{k} $ coming from (\ref{AL:eqn0}) with respect to complex grids $ \{ Z_{l} \}_{l = 1}^{L} $ as given above.

It follows that
\begin{equation*}
    \hat{u}_{test, r}^{k} = \mathcal{G}_{\Theta}f_{test}^{k}.
\end{equation*}
It follows that the testing relative $ l^{2} $ error is given by
\begin{equation*}
    Err_{r} = \frac{1}{K} \sum_{k = 1}^{K} \frac{\sum_{i = 1}^{N} \lvert \hat{u}_{test, r}^{k}(x_{i}) - u_{test}^{k}(x_{i}) \rvert^{2}}{\sum_{i = 1}^{N} \lvert u_{test}^{k}(x_{i}) \rvert^{2}}.
\end{equation*}
We can replace $ \{ x_{i} \} $ by $ \{ y_{j} \} $ to test the error in terms of the generalization of the NO.
\medskip

\subsection{One-Step End-to-End Algorithm.} The only change we make is to replace the pre-training encoder $ W $ by a neural network $ \mathcal{W}_{\theta'} $, therefore (\ref{AL:eqn0}) becomes
\begin{equation}\label{AL:eqn000}
    \mathbf{c}_{0}^{k} : = \frac{1}{L} \sum_{l = 1}^{L} q_{M}(z_{l}) \overline{\mathbf{e}^{a}(z_{l})} \tilde{F}_{N, \theta'}^{k}(z_{l}).
\end{equation}
where
\begin{equation*}
    \tilde{F}_{N, \theta'}^{k}(z) : = \frac{1}{N} \sum_{i = 1}^{N} w_{li, \theta'} A_{a, M}(z, x_{i})f^{(k)}(x_{i}).
\end{equation*}
Everything else are exactly the same as in \S4.3. The encoding neural network only applies to the construction of the initial complex inputs. We then apply the loss (\ref{AL:eqn7}) to optimize both $ \theta' $ and $ \Theta $.
\medskip

\subsection{Summary of Algorithms.}
The two algorithms are summarized below:
\begin{algorithm}[htbp]
\caption{Two-step training of the Segal--Bargmann neural operator}
\label{alg:sbno-two-step}
\raggedright
\textbf{Input:} Training observations
$\{f^{(k)}(x_i),u^{(k)}(x_i)\}_{i,k}$; parameters $a,M,L,S$ with even $L$.
\par\textbf{Output:} Encoder $W$ and operator parameters $\Theta$.
\begin{enumerate}
\renewcommand{\labelenumi}{\arabic{enumi}:}
\setlength{\itemsep}{0pt}
\setlength{\parskip}{0pt}
\item Sample $z_1,\ldots,z_{L/2}$ independently from $\rho_a$ and set
$z_{l+L/2}=\overline{z_l}$.
\item Minimize $L_{W,\mathrm{data}}$ over $W$, freeze $W$, and construct
the encoder $\widetilde B$.
\item Form $\hat u^k=\mathcal G_\Theta f^{(k)}$ using $S$ Segal--Bargmann
layers and minimize $L_{\mathrm{MSE,real}}$ over $\Theta$.
\end{enumerate}
\end{algorithm}

\begin{algorithm}[htbp]
\caption{End-to-end training of the Segal--Bargmann neural operator}
\label{alg:sbno-end-to-end}
\raggedright
\textbf{Input:} Training observations and parameters as above;
encoding network $\mathcal W_{\theta'}$.
\par\textbf{Output:} Trained parameters $\theta'$ and $\Theta$.
\begin{enumerate}
\renewcommand{\labelenumi}{\arabic{enumi}:}
\setlength{\itemsep}{0pt}
\setlength{\parskip}{0pt}
\item Sample conjugate-paired complex points as in
Algorithm~\ref{alg:sbno-two-step}.
\item Construct the initial coefficients $\mathbf c_0^k$ using
$\mathcal W_{\theta'}$; keep the encoders in subsequent layers fixed.
\item Form $\hat u^k=\mathcal G_\Theta f^{(k)}$ and minimize
$L_{\mathrm{MSE,real}}$ jointly over $\theta'$ and $\Theta$.
\end{enumerate}
\end{algorithm}
\medskip

\section{$ L^{\infty} $-Convergence and Independence of The Dimension}
In this section, we consider the convergence rate between the estimate $ \hat{u} $ and the ground truth solution $ u $ of the PDE $ \mathcal{L}u = f $. For notational simplicity, we give proofs only for elliptic smooth differential operator $ \mathcal{L} $. By selecting the appropriate functional spaces, the analysis for time-dependent smooth differential operators is essentially the same, expect notationally more complicated. Recall that we only have finitely many observations $ \{ f(x_{i}) \} $. In \S5.1, we consider whether the unknown ground truth $ F = \mathcal{B}_{a}(f) $ can be approximated by a finite expansion $ \hat{F} = \sum_{\beta \in \Lambda_{M}} \hat{c}_{\beta} e_{\beta}^{a} $ determined by both complex grids $ \{ z_{l} \}_{l = 1}^{L} $ and observations $ \{ f(x_{i}) \}_{i =1 }^{N} $. $ \hat{F} $ serves as an RKHS approximation of $ F $. In \S5.2, we give $ L^{\infty} $-type and $ L^{2} $-type convergence rate between $ \hat{u} $ and $ u $, in terms of the size of the complex grids $ \{ z_{l} \}_{l = 1}^{L} $. Under the stated hypotheses, the $ L^{-\frac{1}{2}} $ reconstruction errors are dimension-free, and are independent of the size of real grids $ \{ x_{i} \}_{i = 1}^{N} $.

Let $ f \in L^{2}(\R^{n}; \R) \cap \calC^{\infty}(\R^{n};\R) $ and suppose that we only know the evaluations of $ f $ at finitely many grids $ \{ x_{i} \}_{i = 1}^{N} $. In this section, we show that how to recover $ f $ from $ \{ f(x_{i}) \}_{i = 1}^{N} $ by introducing complex sampling grids on $ \C^{n} $, provided that a very limited prior knowledge about the evaluation representations $ \{ x_{i} \}_{i = 1}^{N} $ is given.

Theoretically, any $ f \in L^{2}(\R^{n}; \R) $ can be written as
\begin{equation*}
f = \sum_{\beta} c_{\beta} h_{\beta}^{a}, \lVert f \rVert_{L^{2}}^{2} = \sum_{\beta} \lvert c_{\beta} \rvert^{2}.
\end{equation*}
in the $ L^{2} $-sense where
\begin{equation*}
    c_{\beta} = \langle f, h_{\beta}^{a} \rangle_{L^{2}(\R^{n}; \R)}, \mathbf{c} : = (c_{1}, \dotso, c_{\beta}, \dotso, c_{M}).
\end{equation*}
by orthonormality of $ \{ h_{\beta}^{a}  \} $. Numerically, it is always one of the biggest problems of how to project an infinite dimensional space onto a suitable finite dimensional subspace. Practically, we have to choose finitely many elements to form a finite linear combination of $ \{ h_{\beta}^{a} \} $ in the sense that
\begin{equation*}
    P_{M}f : = \sum_{\beta \in \Lambda_{M}} c_{\beta} h_{\beta}^{a}, \lVert f - P_{M} f \rVert_{L^{2}(\R^{n}; \R)} \rightarrow 0 \: {\rm as} \; M \rightarrow \infty.
\end{equation*}
If we have complete information about $ f $, then each $ c_{\beta} = \langle f, h_{\beta}^{a} \rangle_{L^{2}} $ in the sense of $ L^{2} $-inner product. But with only $ \{ f(x_{i}) \}_{i =1}^{N} $, the standard way is to approximate $ f $ by $ \sum_{\beta \in \Lambda_{M}} c_{\beta} h_{\beta}^{a} $ via the discrete least square
\begin{equation}\label{CON:eqnl0a}
\tilde{\mathbf{b}} = \text{argmin}_{\mathbf{d}} \left(\frac{1}{N} \sum_{i = 1}^{N} \left( f(x_{i}) - \sum_{\beta \in \Lambda_{M}} d_{\beta} h_{\beta}^{a}(x_{i}) \right)^{2} \right).
\end{equation}
The other way is to estimate $ c_{\beta} $ via naive discrete sum
\begin{equation}\label{CON:eqnl0b}
   \tilde{\mathbf{c}} = (\tilde{c}_{1}, \dotso, \tilde{c}_{\beta}, \dotso, \tilde{c}_{M}), \tilde{c}_{\beta} = \frac{1}{N} \sum_{i = 1}^{N} f(x_{i}) h_{\beta}^{a}(x_{i})
\end{equation}
This unweighted sum does not in general approximate the Lebesgue inner product without suitable sampling or quadrature weights. Set the $ M \times M $ matrix $ H = (H)_{\beta, \beta'} $ as
\begin{equation}\label{CON:eqnl0}
    H = \begin{pmatrix} \sum_{ i = 1}^{N} h_{1}^{a}(x_{i})h_{1}^{a}(x_{i}) & \dotso & \sum_{i = 1}^{N} h_{1}^{a}(x_{i}) h_{M}^{a}(x_{i}) \\ \vdots & \ddots & \vdots \\ \sum_{i = 1}^{N} h_{M}^{a}(x_{i})h_{1}^{a}(x_{i}) & \dotso & \sum_{i = 1}^{N} h_{M}^{a}(x_{i})h_{M}^{a}(x_{i}) \end{pmatrix}.
\end{equation}
If $ H $ is invertible,
\begin{equation*}
    \tilde{b}_{\beta} = \sum_{\beta' \in \Lambda_{M}} (H^{-1})_{\beta, \beta'} \sum_{i = 1}^{N} f(x_{i})h_{\beta'}^{a}(x_{i}) = \sum_{\beta' \in \Lambda_{M}} N(H^{-1})_{\beta, \beta'} \left( \frac{1}{N}\sum_{i = 1}^{N} f(x_{i})h_{\beta'}^{a}(x_{i})\right) = \sum_{\beta' \in \Lambda_{M}} N(H^{-1})_{\beta, \beta'} \hat{c}_{\beta}'.
\end{equation*}
From this, we get
\begin{equation*}
    \lVert \tilde{\mathbf{c}} - \mathbf{c} \rVert_{\R^{M}} = \left\lVert \frac{H}{N} \tilde{\mathbf{b}} - \mathbf{c} \right\rVert_{\R^{M}} \leqslant \left\lVert \frac{H}{N} (\tilde{\mathbf{b}} - \mathbf{c}) \right\rVert_{\R^{M}} + \left\lVert \left( \frac{H}{N} - I \right) \mathbf{c} \right\rVert_{\R^{M}} 
\end{equation*}
If $ \lVert \tilde{\mathbf{b}} - \mathbf{c} \rVert = O(1) $ and $ \frac{1}{N} \lVert H \rVert = O(1) $, then it implies that $ \lVert \tilde{\mathbf{c}} - \mathbf{c} \rVert = O(1) $. The hypothesis of $ 1/N \lVert H \rVert = O(1) $ is a plausible assumption. If the empirical inner products approximate the $ L^{2} $ inner product $ \langle h_{\beta}^{a}, h_{\beta'}^{a} \rangle_{L^{2}} = \delta_{\beta, \beta'} $. Therefore, there are grids $ \{ x_{i} \} $ such that
\begin{equation*}
    \frac{1}{N} \sum_{i = 1}^{N} h_{\beta}^{a}(x_{i}) h_{\beta'}^{a}(x_{i}) \approx \delta_{\beta, \beta'} \Rightarrow H \approx N I_{M \times M}.
\end{equation*}
Our hypothesis says that the symmetric matrix $ H $ is not too far away from an invertible matrix whose largest eigenvalue is $ O(N) $, i.e. $ H $ is not severely ill-conditioned. Analogously,
\begin{equation*}
    \lVert \tilde{\mathbf{b}} - \mathbf{c} \rVert_{\R^{M}} = \lVert NH^{-1} \tilde{\mathbf{c}} - \mathbf{c} \rVert_{\R^{M}} \leqslant \lVert NH^{-1} (\tilde{\mathbf{c}} - \mathbf{c}) \rVert_{\R^{M}} + \lVert (NH^{-1} - I) \mathbf{c} \rVert_{\R^{M}}.
\end{equation*}
If $ \lVert \tilde{\mathbf{c}} - \mathbf{c} \rVert_{\R^{M}} = O(1) $ and $ N \lVert H^{-1} \rVert = O(1) $, then $ \lVert \tilde{\mathbf{b}} - \mathbf{c} \rVert_{\R^{M}} = O(1) $. 

The discussion above indicates that the matrix $ H $ characterized by the evaluations of orthonormal basis $ \{ h_{\beta}^{a} \} $ is crucial. But it is more important to have prior knowledge of how grids $ \{ x_{i} \} $ are distributed as the basic example above shows. However, the best we can say is that $ \{ x_{i} \} $ are samples of random variables $ \{ X_{i} \} $. We cannot say more about the means and variances of $ \{ X_{i} \} $. 

The above discussion also indicates that the truncation of the orthonormal expansion of $ f $ is also crucial. These motivate the following two assumption:

{\bf{Assumption A}}: Let $ f \in L^{2}(\R^{n}; \R) \cap \calC^{\infty}(\R^{n};\R) $. We say that $ f $ is {\bf{accessible}} if there exists a finite index set $ \Lambda_{M} = \{ \beta \in \mathbb{N}_{0}^{n}: \lvert \Lambda_{M} \rvert = M \} = \{ \beta_{1}, \dotso, \beta_{M} \} $ and associated finite collection of orthonormal basis $ V_{M} = \{ h_{\beta}^{a} : \beta \in \Lambda_{M} \} $ such that
\begin{equation*}
f = \sum_{\beta \in \Lambda_{M}} c_{\beta} h_{\beta}^{a} \Leftrightarrow f \in \text{Span}{V_{M}}
\end{equation*}
in the $ L^{2} $-sense. It is plausible e.g.\ under the low-dimensional sparsity assumption in \cite{CRT}.
\begin{remark}
    The {\it{Assumption A}} is very restrictive. However, we will give dimension-free reconstruction error with this assumption for the purpose of notational simplicity. We will then give results of the reconstruction error in the most general setting, especially without this assumption.
\end{remark}

{\bf{Assumption B}}: Let $ F = \mathcal{B}_{a}(f) $. We say that the information $ f $ is {\bf{recoverable}} from finite observations $ \{ f(x_{i}) \}_{i =1}^{N} $ if for some collection of complex random variables $ \{ Z_{l} \}_{l = 1}^{N} \overset{i.i.d.}{\sim} \rho_{\C}(z) $, there exists a random matrix $ W = (w_{li}) \in \C^{L \times N} $ (called the {\bf{information recovery matrix}}) such that for our complex approximation
\begin{equation}\label{REV2:eqn00}
    \tilde{F}_{N}(Z_{l}) : = \frac{1}{N} \sum_{i = 1}^{N} w_{li} A_{a}(Z_{l}, X_{i})f(X_{i}),
\end{equation}
the difference between the approximation $ \tilde{F}_{N} $ and unknown ground truth $ F $ at $ Z_{l} $
\begin{equation*}
   \eta_{l} : = F(Z_{l}) - \tilde{F}_{N}(Z_{l})
\end{equation*}
satisfies the following conditions, where $ Z = (Z_{1},  \dotso, Z_{L}) $ denotes the full collection of complex samples:
\begin{equation}\label{REV2:eqn0}
\mathbb{E}_{X | Z}[\eta_{l}] = 0, \mathbb{E}_{X | Z}[\eta_{l}\bar{\eta}_{m}] = \sigma_{lm}^{2} \delta_{lm}, \sigma_{lm}^{2} \leqslant (C')^{2} < \infty.
\end{equation}
\begin{remark}\label{REV2:re0}
For a fixed training set, \cite[Theorem 3.2]{DU} proves convergence of the total training loss to zero for a sufficiently wide two-layer ReLU network under its stated assumptions. We modify their proof and show that {\it{Assumption B}} can be achieved under some hypotheses. The proof can be found in Appendix. Therefore, it is reasonable to assume {\it{Assumption B}}.

We state several basic cases when {\it{Assumption B}} are satisfied in Remark \ref{CON:rel0} below. We also give two algorithm for the construction of $ W $ in \S4.3 and \S4.4. We also explain why (\ref{REV2:eqn0}) is a reasonable goal under the Segal-Bargmann framework in \S4.2. 

More importantly, we can treat $ W $ as a neural network for the purpose of information recovery from data in the original space to the information in the Bargmann-Fock space by grids on the phase space from physics point of view. We then give an end-to-end algorithm which combines information recovery and operator learning.
\end{remark}

\begin{remark}\label{CON:rel0}
Suppose that we estimate $ f $ by $ \sum_{\beta \in \Lambda_{M}} \hat{c}_{\beta} h_{\beta}^{a} $ from our observations $ \{ f(x_{i}) \} $. Recall that $ \tilde{c}_{\beta} $ is our naive guess of the ground truth $ c_{\beta} = \langle f, h_{\beta}^{a} \rangle_{L^{2}(\R^{n}; \R)} $. Set
\begin{equation*}
  \hat{f} : = \sum_{\beta \in \Lambda_{M}} \hat{c}_{\beta} h_{\beta}^{a}, \tilde{f} : = \sum_{\beta \in \Lambda_{M}} \tilde{c}_{\beta} h_{\beta}^{a}.
\end{equation*}
We have
\begin{equation*}
\lvert f - \hat{f} \rvert  \leqslant \lvert f - P_{M} f \rvert + \lvert P_{M}f - \tilde{f} \rvert + \lvert \tilde{f} - \hat{f} \rvert.
\end{equation*}
Either the least square (\ref{CON:eqnl0a}) or the approximation of the inner product (\ref{CON:eqnl0b}) recovers the information of the projection of $ f $ onto $ V_{M} $. We definitely loss the information of $ f $ on $ (\text{Span}(V_{M}))^{\perp} $.

The {\it{Assumption A}} is very restrictive, which gives a brutal control of the first term in the inequality above. The convergence $ \lVert f - P_{M}f \rVert_{L^{2}(\R^{n}; \R)} \xrightarrow{M \rightarrow \infty} 0 $ motivates finite-dimensional approximation, although exact finite representability in {\it{Assumption A}} is stronger. But it is not enough since our goal is to give $ L^{\infty} $-norm of $ f - \hat{f} $, as the $ L^{2} $-sense tail estimate usually aligns with the curse of dimensionality. 

In Lemma \ref{CON:lemma0} below, we show that if $ f $ has enough regularity, and for any spatial dimension $ n $ and any truncation $ M $, there exists an appropriate $ a $ such that the $ \sup_{x \in K} \lvert f - P_{M}f \rvert $ is small enough for any compact set $ K $ given a priori.

Note that the Gauss-Hermite polynomials are nonlinear eigenfunctions of the Fourier transform and hence the highest order of a Gauss-Hermite polynomial reflects the highest frequency it represents. In practice, we may always assume that the functions that characterize our observable data have finite maximum frequencies. Note that a finite Hermite expansion in {\it{Assumption A}} is not exactly band-limited.

We still need to control the second and the third terms in the inequality above. The last term may be tackled by different machine learning approaches such as neural networks. 

The gap is the difference between the ground truth $ P_{M} f $ and our data-driven recovery $ \tilde{f} $. If we know the invertibility of $ H $ with respect to $ \{ x_{i} \} $, then we can recover $ \mathbf{c} $ directly. In this case, we get
\begin{equation*}
    w_{li} = \frac{N\mathbf{e}^{a}(Z_{l})^{t}H^{-1}\mathbf{h}^{a}(X_{i})}{A_{a, M}(Z_{l}, X_{i})}, \forall i, l
\end{equation*}
where
\begin{equation*}
\mathbf{e}^{a}(z) = \begin{pmatrix} e_{1}^{a}(z) \\ \vdots \\ e_{M}^{a}(z) \end{pmatrix}, \mathbf{h}^{a}(z) = \begin{pmatrix} h_{1}^{a}(z) \\ \vdots \\ h_{M}^{a}(z) \end{pmatrix}, A_{a, M}(z, x) : =  P_{M}A_{a}(z, x) = \sum_{\beta \in \Lambda_{M}} e_{\beta}^{a}(z) h_{\beta}^{a}(x).
\end{equation*}
When $ \{ x_{i} \} $ are sampling from i.i.d. random variables $ \{ X_{i} \}_{i = 1}^{N} \overset{i.i.d.}{\sim} \rho $, and $ N $ is large enough, {\bf{Assumption B}} holds by taking Monte-Carlo sampling with
\begin{equation*}
    w_{li} = \frac{1}{\rho(X_{i})}, \forall l.
\end{equation*} 
In practice, these are not the cases in general. However, the two basic examples above illustrate the feasibility of our {\it{Assumption B}}. This is one of our main purposes of the introduction of the $ a $-Segal-Bargmann transform and the $ a $-Bargmann-Fock space, where we use complex samples to try to construct a latent representation of the original data, in order to obtain dimension reduction and feature compression to some extent. {\it{Assumption B}} states recovery conditions used in our theoretical derivations.

It is easier to get error bounds of $ \lVert f- \tilde{f} \rVert_{L^{2}(\R^{n}; \R)} $, but our goal is $ \lVert f - \tilde{f} \rVert_{L^{\infty}(\R^{n}; \R)} $. The Segal-Bargmann transform allows us to take advantage of the RKHS structure and the rigidity of complex analyticity from the Bargmann-Fock space.

Later we would also like to discuss how to recover $ \mathbf{c} $ by dropping the {\it{Assumption A}}.
\end{remark}
\subsection{Convergence with Complex $ L $-Sampling}
The results in this subsection are preliminaries for our neural operator learning process in later sections.

In this subsection, we will show that for $ f $ that satisfies {\it{Assumption A}} and the grids $ \{ x_{i} \} $ that satisfy {\it{Assumption B}}, we can apply the $ a $-Segal-Bargmann transform, and introduce complex $ L $-sampling in the $ a $-Bargmann-Fock space to obtain an error bound of
\begin{equation*}
\sup_{x \in \R^{n}} \lvert f - \hat{f} \rvert
\end{equation*}
for $ f \in L^{2}(\R^{n}; \R) \cap \calC^{\infty}(\R^{n}; \R) $. If we assume that $ f \in H^{s}(\R^{n}; \R) $ for $ s > \frac{n}{2} $, we can get the same error bound without {\it{Assumption A}}. 

More importantly, we show that with an appropriate choice of $ a $ that is specific for the spatial dimension $ n $, the error bound we will give is independent of the spatial dimension $ n $.

The crucial step is to not to consider the error bound in the original space, but in the $ a $-Bargmann-Fock space. We see the advantage of the $ a $-Segal-Bargmann transform immediately. By Definition \ref{SB:def2} and Definition \ref{SB:def3}, we have
\begin{equation*}
    \mathcal{B}_{a}(f) = F = \sum_{\beta} c_{\beta}e_{\beta}^{a}, c_{\beta} = \langle f, h_{\beta}^{a} \rangle_{L^{2}(\R^{n}; \R)} = \langle F, e_{\beta}^{a} \rangle_{\mathcal{F}_{a}^{2}(\C^{n})} 
\end{equation*}
where the equality holds pointwise! In addition,
\begin{equation*}
   \hat{F} : = \sum_{\beta \in \Lambda_{M}} \hat{c}_{\beta} e_{\beta}^{a} = \mathcal{B}_{a}(\hat{f}), \lVert f - \hat{f} \rVert_{L^{2}(\R^{n}; \R)} = \lVert F - \hat{F} \rVert_{\mathcal{F}_{a}^{2}(\C^{n})}.
\end{equation*}
It follows that we can equally discuss $ \lVert F - \hat{F} \rVert_{\mathcal{F}_{a}^{2}(\C^{n})} $. More importantly, merely with {\it{Assumption A}} and {\it{Assumption B}}, we can give pointwise error bound
\begin{equation*}
    \sup_{z \in D} \lvert F - \hat{F} \rvert
\end{equation*}
which is again independent of the dimension, since the Fock space $ \mathcal{F}_{a}^{2} $ is an RKHS. The key point is that we will estimate $ \hat{c}_{\beta} $ by using complex grids $ \{ z_{l} \} \subset \C^{n} $, which are decoupled from the original grids $ \{ x_{i} \} $. The decoupling between real and complex grids allows us to choose complex grids at will.

We need the following technical lemma for the estimate of the projection of $ f $ onto $ \text{Span}(V_{M})^{\perp} $.
\begin{lemma}\label{CON:lemma0}
Let $ f \in H_{a}^{s}(\R^{n}; \R) \cap C^{\infty}(\R^{n}; \R) $ for some $ s \geq s_{0} > \frac{n}{2} $. Then for each $ M_{0} \in \mathbb{N} $,
\begin{equation}\label{CON:eqnl2}
\sup_{x \in \R^{n}} \lvert f - P_{M_{0}} f \rvert = \sup_{x \in \R^{n}} \lvert \sum_{\lvert \beta \rvert > M_{0}} c_{\beta} h_{\beta}^{a} \rvert \leqslant (4\pi)^{-\frac{n}{4}} \sqrt{\frac{\Gamma(s - \frac{n}{2})}{\Gamma(s)}} \lVert f \rVert_{\mathcal{H}_{a}^{s}}
\end{equation}
and
\begin{equation}\label{CON:eqnl1}
\sup_{x \in \R^{n}} \lvert f - P_{M_{0}} f \rvert \leqslant (1 + 2a(2M_{0} + n))^{-\frac{s - s_{0}}{2}} (4\pi)^{-\frac{n}{4}} \sqrt{\frac{\Gamma(s_{0} - \frac{n}{2})}{\Gamma(s_{0})}} \lVert f \rVert_{\mathcal{H}_{a}^{s}}.
\end{equation}
In addition,
\begin{equation}\label{CON:eqnl3}
\lVert f - P_{M_{0}} f \rVert_{\calL^{2}(\R^{n}; \R)} \leqslant (1 + 2a(2M_{0} + n))^{-\frac{s}{2}} \lVert f \rVert_{\mathcal{H}_{a}^{s}(\R^{n}; \R)}.
\end{equation}
\end{lemma}
\begin{proof}
By Definition \ref{Def5} and the eigenvalue equation of $ \mathcal{H}_{\beta}^{a} $, we compute directly that
\begin{align*}
    \lvert f(x) - P_{M_{0}}f(x) \rvert & = \lvert \sum_{\lvert \beta \rvert > M_{0}} c_{\beta} h_{\beta}^{a}(x) \rvert \leqslant \sum_{\lvert \beta \rvert > M_{0}} \left( (1 + \lambda_{\beta, n} )^{\frac{s}{2}} \lvert c_{\beta} \rvert \right) \left( (1 + \lambda_{\beta, n})^{-\frac{s}{2}} \lvert h_{\beta}^{a}(x) \rvert \right) \\
    & \leqslant \left( \sum_{\lvert \beta \rvert > M_{0}} (1 + \lambda_{\beta, n} )^{s} \lvert c_{\beta} \rvert^{2} \right)^{\frac{1}{2}} \left( \sum_{\lvert \beta \rvert > M_{0}} (1 + \lambda_{\beta, n})^{-s} \lvert h_{\beta}^{a}(x) \rvert^{2} \right)^{\frac{1}{2}} \\
    & \leqslant \lVert f \rVert_{\mathcal{H}_{a}^{s}(\R^{n}; \R)}  \left( \sum_{\lvert \beta \rvert > M_{0}} (1 + \lambda_{\beta, n})^{-s} \lvert h_{\beta}^{a}(x) \rvert^{2}\right)^{\frac{1}{2}}.
\end{align*}
Here we applied Cauchy-Schwarz in the second line above.

When $ s > \frac{n}{2} $, we apply the Gamma function identity
\begin{align*}
   \sum_{\beta} (1 + \lambda_{\beta, n})^{-s} \lvert h_{\beta}^{a}(x) \rvert^{2} & = \sum_{\beta} \frac{1}{\Gamma(s)} \int_{0}^{\infty} t^{s - 1} e^{-(1 + \lambda_{\beta,n})t} \lvert h_{\beta}^{a}(x) \rvert^{2} dt \\
   & = \frac{1}{\Gamma(s)} \int_{0}^{\infty} t^{s - 1} e^{-t} \left( \sum_{\beta} e^{-\lambda_{\beta, n}t} \lvert h_{\beta}^{a}(x) \rvert^{2} \right) dt.
\end{align*}
Here the sum and the integration are interchangeable by Tonelli's theorem. By the one dimensional Mehler's formula \cite[(18.18.28)]{THAN} and the multi-dimensional derivation \cite[(4.9)]{KTX}, we have
\begin{equation*}
    \sum_{\beta} e^{-\lambda_{\beta, n}t} \lvert h_{\beta}^{a}(x) \rvert^{2} = \left(\frac{a}{\pi \sinh (4at)} \right)^{\frac{n}{2}} e^{-2a \tanh (2at) \lvert x \rvert^{2}}
\end{equation*}
which is the Mehler kernel as the diagonal heat kernel with respect to $ e^{-t \mathcal{H}_{\beta}^{a}} $. It follows that
\begin{align*}
\sup_{x \in \R^{n}} \left( \sum_{\lvert \beta \rvert > M_{0}} (1 + \lambda_{\beta, n})^{-s} \lvert h_{\beta}^{a}(x) \rvert^{2} \right) & \leqslant \frac{1}{\Gamma(s)} \int_{0}^{\infty} t^{s - 1}e^{-t} \left( \frac{a}{\pi \sinh(4at)} \right)^{\frac{n}{2}} \left( \sup_{x \in \R^{n}} e^{-2a \tanh (2at) \lvert x \rvert^{2}} \right) dt \\
& \leqslant \frac{1}{\Gamma(s)} \int_{0}^{\infty} t^{s - 1}e^{-t} \left( \frac{a}{\pi \sinh(4at)} \right)^{\frac{n}{2}} dt \\
& \leqslant \frac{1}{\Gamma(s)} \int_{0}^{\infty} t^{s - 1}e^{-t} \left( \frac{a}{4at\pi} \right)^{\frac{n}{2}} dt = \frac{1}{(4\pi)^{\frac{n}{2}}\Gamma(s)} \int_{0}^{\infty} t^{s - \frac{n}{2} - 1} e^{-t} dt \\
& = (4\pi)^{-\frac{n}{2}} \frac{\Gamma(s - \frac{n}{2})}{\Gamma(s)}.
\end{align*}
The ratio $ \frac{\Gamma(s - \frac{n}{2})}{\Gamma(s)} \leq 1 $ if $ s - \frac{n}{2} \geq 2 $. Therefore we conclude that
\begin{equation*}
    \sup_{x \in \R^{n}} \lvert f(x) - P_{M_{0}}f(x) \rvert \leqslant (4\pi)^{-\frac{n}{4}} \sqrt{\frac{\Gamma(s - \frac{n}{2})}{\Gamma(s)}} \lVert f \rVert_{\mathcal{H}_{a}^{s}(\R^{n}; \R)},
\end{equation*}
which gives (\ref{CON:eqnl2}). By the same argument as above,
\begin{align*}
    \lvert f(x) - P_{M_{0}}f(x) \rvert & = \lvert \sum_{\lvert \beta \rvert > M_{0}} c_{\beta} h_{\beta}^{a}(x) \rvert \leqslant \sum_{\lvert \beta \rvert > M_{0}} \left( (1 + \lambda_{\beta, n} )^{\frac{s}{2}} \lvert c_{\beta} \rvert \right) \left( (1 + \lambda_{\beta, n})^{-\frac{s}{2}} \lvert h_{\beta}^{a}(x) \rvert \right) \\
    & \leqslant \left( \sum_{\lvert \beta \rvert > M_{0}} (1 + \lambda_{\beta, n} )^{s} \lvert c_{\beta} \rvert^{2} \right)^{\frac{1}{2}} \left( \sum_{\lvert \beta \rvert > M_{0}} (1 + \lambda_{\beta, n})^{-s} \lvert h_{\beta}^{a}(x) \rvert^{2} \right)^{\frac{1}{2}} \\
    & \leqslant \lVert f \rVert_{\mathcal{H}_{a}^{s}(\R^{n}; \R)}  \left( \sum_{\lvert \beta \rvert > M_{0}} (1 + \lambda_{\beta, n})^{-(s - s_{0})} (1 + \lambda_{\beta, n})^{s_{0}} \lvert h_{\beta}^{a}(x) \rvert^{2}\right)^{\frac{1}{2}}
\end{align*}
Note that $ \lambda_{\beta, n} = (4a \lvert \beta \rvert + 2an) $, it follows that when $ \lvert \beta \rvert > M_{0} $,
\begin{equation*}
     \sup_{x \in \R^{n}} \lvert f(x) - P_{M_{0}}f(x) \rvert \leqslant (1 + 2a(2M_{0} + n))^{-\frac{s - s_{0}}{2}} (4\pi)^{-\frac{n}{4}} \sqrt{\frac{\Gamma(s_{0} - \frac{n}{2})}{\Gamma(s_{0})}}  \lVert f \rVert_{\mathcal{H}_{a}^{s}(\R^{n}; \R)},
\end{equation*}
which gives (\ref{CON:eqnl1}).

The last inequality (\ref{CON:eqnl3}) is more straightforward,
\begin{align*}
    \lVert f - P_{M_{0}} f \rVert_{L^{2}(\R^{n}; \R)}^{2} & = \sum_{\lvert \beta \rvert > M_{0}} \lvert c_{\beta} \rvert^{2} \\
    & = \sum_{\lvert \beta \rvert > M_{0} }(1 + 4a \lvert \beta \rvert + 2an)^{-s}(1 + 4a \lvert \beta \rvert + 2an)^{s} \lvert c_{\beta} \rvert^{2} \\
    & \leqslant \frac{1}{(1 + 2a(2M_{0} + n))^{s}} \sum_{\beta}(1 + 4a \lvert \beta \rvert + 2an)^{s} \lvert c_{\beta} \rvert^{2} \\
    & = \frac{1}{(1 + 2a(2M_{0} + n))^{s}} \lVert f \rVert_{\mathcal{H}_{a}^{s}(\R^{n}; \R)}^{2}.
\end{align*}
\end{proof}
\medskip

Denote by
\begin{equation*}
    \hat{F}(z) : = \sum_{\beta \in \Lambda_{M}} \hat{c}_{\beta} e_{\beta}^{a}(z)
\end{equation*}
the data-driven approximation of $ F $. Let's focus on $ \sup_{x \in D} \lvert F(z) - \hat{F}(z) \rvert $ first for some bounded, open subset $ D \subset \C^{n} $. By Liouville's theorem, the error $ F - \hat{F} $ cannot be bounded on the whole space $ \C^{n} $ unless it is constant. The other error bounds, including our main goal $ \sup_{x \in \R^{n}} \lvert f(x) - \hat{f}(x) \rvert $ are natural consequences of the following argument. 

With {\bf{Assumption A}} and {\bf{Assumption B}}, we are ready to give a dimension-free error bound for $ \lvert F(z) - \hat{F}(z) \rvert $ for all $ z \in D $.
\begin{theorem}\label{CON:thml1}
Let $ f \in L^{2}(\R^{n}; \R) $, $ F = \mathcal{B}_{a}(f) $, $ \lVert f \rVert_{L^{2}(\R^{n}; \R)} \leqslant C < \infty $ and $ \lVert f \rVert_{L^{\infty}(\R^{n}; \R)} = C' < \infty $. Let $ D \subset \C^{n} $ be a domain such that $ \bar{D} $ is compact. If both {\it{Assumption A}} and {\it{Assumption B}} holds, then for each fixed dimension $ n $ and fixed truncation $ M = \lvert \Lambda_{M} \rvert $, the data-driven approximation $ \hat{F}(z) = \sum_{\beta \in \Lambda_{M}} \hat{c}_{\beta} e_{\beta}^{a}(z) $ satisfies
\begin{equation}\label{CON:eqnl5}
\mathbb{E}_{Z, X}\big[\sup_{z \in D}\lvert F(z) - \hat{F}(z) \rvert \big] = O(L^{-\frac{1}{2}}),
\end{equation}
where $ \hat{\mathbf{c}} = (\hat{c}_{1}, \dotso, \hat{c}_{\beta}, \dotso, \hat{c}_{M}) $ is determined by $ \{ f(x_{i}) \} $ and complex grids $ \{ z_{l} \}_{l = 1}^{L} $ sampled random variables $ Z_{1}, \dotso, Z_{L} \overset{i.i.d.}{\sim} \rho_{a}(z) = \frac{\sum_{\beta \in \Lambda_{M}} \lvert e_{\beta}^{a}(z) \rvert^{2}}{M} \mu_{a}(z) $. Precisely speaking,
\begin{equation}\label{REV:eqn1}
\mathbb{E}_{X, Z}\big[\sup_{z \in D}\lvert F(z) - \hat{F}(z) \rvert \big] \leqslant \sup_{z \in D} (e^{\frac{a\lvert z \rvert^{2}}{2}}) \cdot \sqrt{\frac{(M - 1)C^{2}}{L} + \frac{M(C')^{2}}{L}}.
\end{equation}
\end{theorem}
\begin{proof}
By {\it{Assumption A}}, $ f = P_{M} f $. By the linearity of the $ a $-Segal-Bargman transform, 
\begin{equation*}
    F(z) = \sum_{\beta \in \Lambda_{M}} c_{\beta} e_{\beta}^{a} : = P_{M}F, c_{\beta} = \langle F, e_{\beta}^{a} \rangle_{\mathcal{F}_{a}^{2}(\C^{n})} = \langle f, h_{\beta}^{a} \rangle_{L^{2}(\R^{n}; \R)}.
\end{equation*}
with pointwise identification. For each fixed $ z $, $ A_{a}(z, x) $ in Definition \ref{SB:de1} is an $ L^{2} $-function. By Definition \ref{SB:de1} and Definition \ref{Def4}), it follows that for fixed z,
\begin{equation*}
    A_{a}(z, x) = \sum_{\beta} \langle A_{a}(z, \cdot), h_{\beta}^{a}(\cdot) \rangle_{L^{2}(\R^{n}; \R)} h_{\beta}^{a} = \sum_{\beta} B_{a}(h_{\beta}^{a}) h_{\beta}^{a} = \sum_{\beta} e_{\beta}^{a} h_{\beta}^{a}.
\end{equation*}
Denote by
\begin{equation*}
   A_{a, M}(z, x) : =  P_{M}A_{a}(z, x) = \sum_{\beta \in \Lambda_{M}} e_{\beta}^{a}(z) h_{\beta}^{a}(x).
\end{equation*}
With $ F = P_{M} F $, we get from uniqueness of orthonormal expansion that
\begin{align*}
   \sum_{\beta \in \Lambda_{M}} c_{\beta} e_{\beta}^{a}(z) & =  F(z) = \int_{\R^{n}} A_{a}(z, x) f(x) dx = \int_{\R^{n}} \left( \sum_{\beta} h_{\beta}^{a}(x) e_{\beta}^{a}(z) \right) f(x) dx \\
   & = \int_{\R^{n}} \left( \sum_{\beta \in \Lambda_{M}} h_{\beta}^{a}(x) e_{\beta}^{a}(z) \right) f(x) dx = \int_{\R^{n}} A_{a, M}(z, x) f(x) dx.
\end{align*}
To construct the estimate $ \hat{F} $ with coefficients $ \hat{\mathbf{c}} \in \C^{M} $, we select complex sampling $ \{ Z_{l} \}_{l = 1}^{L} $ and associated real sampling $ \{ X_{i, l} \}_{i= 1}^{J} $ for each $ l $. 

Set
\begin{equation*}
    K_{M}(z) : = \sum_{\beta \in \Lambda_{M}} \lvert e_{\beta}^{a}(z) \rvert^{2}, d\rho_{a}(z) = \frac{K_{M}(z)}{M} d\mu_{a}(z),
\end{equation*}
we pick
\begin{equation*}
    Z_{1}, \dotso, Z_{L} \overset{i.i.d.}{\sim} \rho_{a}.
\end{equation*}
Such a sampling approach is called Christoffel-function sampling following Cohen-Migliorati's \cite{CM} optimal weighted least-squares sampling, as a unbiased sampling with a uniform control of the variance. Our complex grids $ \{ z_{l} \}_{l = 1}^{L} $ samples from $ \{ Z_{l} \}_{l = 1}^{L} $.

Let $ \{x_{i} \}_{i = 1}^{N} $ be the real grids that are sampling from random variables $ \{ X_{i} \}_{i = 1}^{N} $. By {\it{Assumption B}}, (\ref{REV2:eqn00}) and (\ref{REV2:eqn0}), we have a random matrix $ W $ such that
\begin{equation*}
F = \mathcal{B}_{a}(f) \approx \tilde{F}_{N}(Z_{l}) : =  \frac{1}{N} \sum_{i = 1}^{N} w_{li} A_{a, M}(Z_{l}, X_{i})f(X_{i}) \Rightarrow F(Z_{l}) : = \tilde{F}_{N}(Z_{l}) + \eta_{l}
\end{equation*}
where
\begin{equation*}
\mathbb{E}_{X|Z}[\eta_{l}] = \mathbb{E}_{X|Z}[F(Z_{l}) - \tilde{F}_{N}(Z_{l})] = 0,  \mathbb{E}_{X |Z}[\eta_{l}\bar{\eta}_{m}] = \sigma_{lm}^{2} \delta_{lm}. 
\end{equation*}
Set the weight, the basis vector and the unknown ground truth 
\begin{equation*}
q_{M}(z) = \frac{M}{K_{M}(z)}, \mathbf{e}^{a}(z) = \begin{pmatrix} e_{1}^{a}(z) \\ \vdots \\ e_{M}^{a}(z) \end{pmatrix}, \mathbf{c} = \begin{pmatrix} c_{1} \\ \vdots \\ c_{M} \end{pmatrix}.
\end{equation*}
We define
\begin{equation}\label{REV2:eqn1}
    \hat{\mathbf{c}} : = \frac{1}{L} \sum_{l = 1}^{L} q_{M}(Z_{l}) \overline{\mathbf{e}^{a}(Z_{l})} \tilde{F}_{N}(Z_{l})
\end{equation}
Observe that
\begin{equation*}
    F(Z_{l}) = \mathbf{e}^{a}(Z_{l})^{t} \mathbf{c},
\end{equation*}
(\ref{REV2:eqn1}) implies that
\begin{equation}\label{REV2:eqn2}
\hat{\mathbf{c}} - \mathbf{c} = \frac{1}{L} \sum_{l = 1}^{L} \left( q_{M}(Z_{l})\overline{\mathbf{e}^{a}(Z_{l})}\mathbf{e}^{a}(Z_{l})^{t} - I_{M} \right) \mathbf{c} - \frac{1}{L} \sum_{l = 1}^{L} q_{M}(Z_{l})\overline{\mathbf{e}^{a}(Z_{l})} \eta_{l}
\end{equation}
We now analyze $ \mathbb{E}_{Z, X}\left[\lVert \hat{\mathbf{c}} - \mathbf{c} \rVert^{2} \right] $ by analyzing two terms in (\ref{REV2:eqn2}), respectively. For the first term of (\ref{REV2:eqn2}), note that
\begin{align*}
\mathbb{E}_{Z}\left[ q_{M}(Z_{l})\overline{\mathbf{e}^{a}(Z_{l})}\mathbf{e}^{a}(Z_{l})^{t}\right] & = \int_{\C^{n}} \frac{M}{K_{M}(z)} \overline{\mathbf{e}^{a}(z)}\mathbf{e}^{a}(z)^{t} d\rho_{a}(z) \\
& = \int_{\C^{n}} \frac{M}{K_{M}(z)} \overline{\mathbf{e}^{a}(z)}\mathbf{e}^{a}(z)^{t} \frac{K_{M}(z)}{M} d\mu_{a}(z) = I_{M}
\end{align*}
by orthonormality of $ \{ e_{\beta}^{a} \} $. It follows that
\begin{align*}
    \mathbb{E}_{Z, X}\left[ \frac{1}{L} \sum_{l = 1}^{L} \left( q_{M}(Z_{l})\overline{\mathbf{e}^{a}(Z_{l})}\mathbf{e}^{a}(Z_{l})^{t} - I_{M} \right) \mathbf{c} \right] = \frac{1}{L} \sum_{l = 1}^{L} \mathbb{E}_{Z}\left[ q_{M}(Z_{l})\overline{\mathbf{e}^{a}(Z_{l})}\mathbf{e}^{a}(Z_{l})^{t} - I_{M} \right] \mathbf{c} = 0.
\end{align*}
Since $ \{ Z_{l} \} \overset{i.i.d}{\sim} \rho_{a} $, it follows that
\begin{align*}
   &  \mathbb{E}_{Z, X} \left[ \left\lVert \frac{1}{L} \sum_{l = 1}^{L} \left( q_{M}(Z_{l})\overline{\mathbf{e}^{a}(Z_{l})}\mathbf{e}^{a}(Z_{l})^{t} - I_{M} \right) \mathbf{c} \right\rVert^{2} \right] \\
   & \qquad = \frac{1}{L} \mathbb{E}_{Z} \left[ \left\lVert \left( q_{M}(Z)\overline{\mathbf{e}^{a}(Z)}\mathbf{e}^{a}(Z)^{t} - I_{M} \right) \mathbf{c} \right\rVert^{2} \right], Z \sim \rho_{a} \\
    & \qquad \leqslant \frac{1}{L} \left\lVert \mathbb{E}_{Z}\left[ q_{M}(Z)\overline{\mathbf{e}^{a}(Z)}\mathbf{e}^{a}(Z)^{t}q_{M}(Z)\overline{\mathbf{e}^{a}(Z)}\mathbf{e}^{a}(Z)^{t} \right] - 2 \mathbb{E}_{Z}\left[ q_{M}(Z)\overline{\mathbf{e}^{a}(Z)}\mathbf{e}^{a}(Z)^{t} \right]I_{M} + I_{M}^{2} \right\rVert \lVert \mathbf{c} \rVert^{2} \\
    & \qquad = \frac{1}{L} \left\lVert \mathbb{E}_{Z} \left[ q_{M}(Z)^{2} \overline{\mathbf{e}^{a}(Z)} K_{M}(Z) \mathbf{e}^{a}(Z)^{t} \right] - I_{M} \right\rVert \lVert \mathbf{c} \rVert^{2} \\
    & \qquad = \frac{1}{L} \left\lVert M \mathbb{E}_{Z} \left[ q_{M}(Z) \overline{\mathbf{e}^{a}(Z)} \mathbf{e}^{a}(Z)^{t} \right] - I_{M} \right\rVert \lVert \mathbf{c} \rVert^{2} = \frac{1}{L} \left\lVert (M - 1) I_{M} \right\rVert \lVert \mathbf{c} \rVert^{2} \\
    & \qquad = \frac{M - 1}{L} \lVert f \rVert_{L^{2}(\R^{n}; \R)}^{2} \leqslant \frac{(M - 1)C^{2}}{L}.
\end{align*}
By {\it{Assumption B}}, the second term of (\ref{REV2:eqn2}) satisfies
\begin{align*}
    \mathbb{E}_{Z, X}\left[ \frac{1}{L} \sum_{l = 1}^{L} q_{M}(Z_{l})\overline{\mathbf{e}^{a}(Z_{l})} \eta_{l} \right] & = \mathbb{E}_{Z} \frac{1}{L} \sum_{l = 1}^{L} \left[ q_{M}(Z_{l})\overline{\mathbf{e}^{a}(Z_{l})} \mathbb{E}_{X|Z} \left[ \eta_{l} \right] \right] = 0.
\end{align*}
The second variation satisfies
\begin{align*}
 & \mathbb{E}_{Z, X}\left[ \left\lVert \frac{1}{L} \sum_{l = 1}^{L} q_{M}(Z_{l})\overline{\mathbf{e}^{a}(Z_{l})} \eta_{l} \right\rVert^{2} \right] \\
 & \qquad = \frac{1}{L^{2}}\mathbb{E}_{Z}\left[ \sum_{l, m = 1}^{L} q_{M}(Z_{l})q_{M}(Z_{m})\mathbf{e}^{a}(Z_{l})^{t}\overline{\mathbf{e}^{a}(Z_{m})} \mathbb{E}_{X|Z}\left[ \eta_{l}\bar{\eta}_{m}\right] \right] \\
 & \qquad = \frac{1}{L^{2}} \mathbb{E}_{Z}\left[\sum_{l = 1}^{L} q_{M}(Z_{l})^{2} K_{M}(Z_{l}) \sigma_{ll}^{2} \right] \\
 & \qquad \leqslant \frac{1}{L^{2}} \cdot ML \cdot (C^{'})^{2} = \frac{M(C')^{2}}{L}.
\end{align*}
Therefore, 
\begin{align*}
    \mathbb{E}_{Z, X}\left[\lVert \hat{\mathbf{c}} - \mathbf{c} \rVert^{2} \right] & = \mathbb{E}_{Z, X} \left[ \left\lVert \frac{1}{L} \sum_{l = 1}^{L} \left( q_{M}(Z_{l})\overline{\mathbf{e}^{a}(Z_{l})}\mathbf{e}^{a}(Z_{l})^{t} - I_{M} \right) \mathbf{c} \right\rVert^{2} \right] \\
    & \qquad + \mathbb{E}_{Z, X}\left[ \left\lVert \frac{1}{L} \sum_{l = 1}^{L} q_{M}(Z_{l})\overline{\mathbf{e}^{a}(Z_{l})} \eta_{l} \right\rVert^{2} \right] \\
    & \leqslant \frac{(M - 1)C^{2}}{L} +  \frac{M(C')^{2}}{L}.
\end{align*}
Applying the RKHS inequality, it follows immediately that
\begin{align*}
    \mathbb{E}_{X, Z}\big[\sup_{z \in D}\lvert F(z) - \hat{F}(z) \rvert \big] & \leqslant \mathbb{E}_{X, Z}\big[\sup_{z \in D}\sqrt{K_{a}(z, z)}\lVert F - \hat{F} \rVert_{\mathcal{F}_{a}^{2}(\C^{n})} \big] \\
    & \leqslant \sup_{z \in D} (e^{\sqrt{a\lvert z \rvert^{2}}}) \left( \mathbb{E}_{X, Z} \left[ \lVert \mathbf{c} - \hat{\mathbf{c}} \rVert^{2} \right] \right)^{\frac{1}{2}}  \\
    & \leqslant \sup_{z \in D} (e^{\frac{a\lvert z \rvert^{2}}{2}}) \cdot \sqrt{\frac{(M - 1)C^{2}}{L} +  \frac{M(C')^{2}}{L}}.
\end{align*}
This gives (\ref{REV:eqn1}). For fixed $ M $ (e.g. when $ \Lambda_{M} $ is a sparse set in the sense of \cite{CRT}), we conclude that
\begin{equation*}
\mathbb{E}_{Z, X}\big[\sup_{z \in D}\lvert F(z) - \hat{F}(z) \rvert \big] = O(L^{-\frac{1}{2}}),
\end{equation*}
which is (\ref{CON:eqnl5}).
\end{proof}
\begin{remark}\label{REV:re1}
The error bound (\ref{REV:eqn1}) depends only on $ f $, $ \{ x_{i} \} $, $ M $, $ L $, and is independent of the spatial dimension $ n $.

More importantly, we achieve the canonical Monte-Carlo rate $ O(L^{-1/2}) $ on the compact set $ \bar{D} \subset \C^{n} $ for the expected uniform error, not just root-mean-square error.

This pointwise error bound on $ D \subset \C^{n} $ leads to the pointwise error bound of $ f - \hat{f} $ on compact sets $ K \subset \R^{n} $, provided that {\it{Assumption A}} and {\it{Assumption B}} hold.
\end{remark}
Now we are ready to give a dimension-free pointwise error bound between $ f $ and $ \hat{f} $ within any compact subset $ K \subset \R^{n} $. However, if we compute $ \lvert f - \hat{f} \rvert $ without using the Segal-Bargmann transform and our estimate of $ \hat{\mathbf{c}} $ by complex data, we will have no control, even in the $ L^{2} $-sense, as {\it{Assumption B}} does not tell us how to recover information merely with data in $ L^{2}(\R^{n}; \R) $; or we might get an error bound suffers the curse of dimensionality. Note that {\it{Assumption A}} implies $ f \in \calC^{\infty}(\R^{n}; \R) $ immediately; therefore we can write $ f = \sum_{\beta \in \Lambda_{M}} c_{\beta} h_{\beta}^{a} $ in the pointwise sense. 
\begin{theorem}\label{REV:thm2}
Let $ f \in L^{2}(\R^{n}; \R) $, $ \lVert f \rVert_{L^{2}(\R^{n}; \R)}^{2} \leqslant C < \infty $ and $ \lVert f \rVert_{L^{\infty}(\R^{n}; \R)} = C' < \infty $. Let $ D = B_{\C^{n}}(0; R) \subset \C^{n} $ and $ K = \bar{B}_{\R^{n}}(0; r) \subset Re D = B_{\R^{n}}(0; R) \subset D $. If both {\it{Assumption A}} and {\it{Assumption B}} hold with respect to random variables $ \{ Z_{l} \}_{l = 1}^{L} \overset{i.i.d.}{\sim} \rho_{a}(z) $, then for each fixed dimension $ n $ and fixed truncation $ M = \lvert \Lambda_{M} \rvert $, there exist an appropriate choice of $ a $ and a data-driven approximation $ \hat{f}(x) = \sum_{\beta \in \Lambda_{M}} \hat{c}_{\beta} h_{\beta}^{a}(x) $ such that the dimension-free algebraic error rate
\begin{equation}\label{REV:eqn3}
\mathbb{E}_{Z, X}\big[\sup_{x \in K}\lvert f(x) - \hat{f}(x) \rvert \big] = O(L^{-\frac{1}{2}}),
\end{equation}
holds, where $ \hat{\mathbf{c}} = (\hat{c}_{1}, \dotso, \hat{c}_{\beta}, \dotso, \hat{c}_{M}) $ is determined by complex grids $ \{ z_{l} \}_{l = 1}^{L} $ that are sampling from $ \{ Z_{l} \}_{l = 1}^{L} $. Precisely speaking,
\begin{equation}\label{REV:eqn4}
\begin{split}
& \mathbb{E}_{Z, X}\big[\sup_{x \in K}\lvert f(x) - \hat{f}(x) \rvert \big] \\
& \qquad \leqslant \left( \frac{2a}{\pi} \right)^{\frac{n}{4}} \bigg( e^{\frac{aR^{2}}{2}} \cdot \frac{(aR^{2})^{n}}{n!}  \\
& \qquad \qquad + e^{-a(R - r)^{2}} s^{-\frac{n}{2}} (1 - s)^{-\frac{n + B_{M}}{2}} \sqrt{\frac{\Gamma(n, saR^{2})}{\Gamma(n)}\cdot \frac{\Gamma\left(n + B_{M}, (1 - s)aR^{2} \right)}{\Gamma(n + B_{M})}}\bigg) \\
& \qquad \qquad \qquad \times \sqrt{\frac{(M - 1)C^{2}}{L} +  \frac{M(C')^{2}}{L}}.
\end{split}
\end{equation}
\end{theorem}
\begin{proof}
We apply the inverse Segal-Bargmann transform
\begin{align*}
    \lvert f(x) - \hat{f}(x) \rvert & = \lvert \mathcal{B}_{a}^{-1}(F - \hat{F}) \rvert \leqslant \left( \frac{2a}{\pi} \right)^{\frac{n}{4}} \left( \frac{a}{\pi} \right)^{n} \int_{\C^{n}} \lvert e^{-\frac{a}{2} \bar{z} \cdot \bar{z} -a \lvert x \rvert^{2} + 2a \bar{z} \cdot x } \rvert \lvert F - \hat{F} \rvert e^{-a\lvert z \rvert^{2}} d\mu(z) \\
    & = \left( \frac{2a}{\pi} \right)^{\frac{n}{4}} \left( \frac{a}{\pi} \right)^{n} \int_{\lvert z \rvert \leq R} \lvert e^{-\frac{a}{2} \bar{z} \cdot \bar{z} -a \lvert x \rvert^{2} + 2a \bar{z} \cdot x } \rvert \lvert F - \hat{F} \rvert e^{-a\lvert z \rvert^{2}} d\mu(z) \\
    & \qquad + \left( \frac{2a}{\pi} \right)^{\frac{n}{4}} \left( \frac{a}{\pi} \right)^{n} \int_{\lvert z \rvert > R} \lvert e^{-\frac{a}{2} \bar{z} \cdot \bar{z} -a \lvert x \rvert^{2} + 2a \bar{z} \cdot x } \rvert \lvert F - \hat{F} \rvert e^{-a\lvert z \rvert^{2}} d\mu(z) \\
    & : = E_{1}(x) + E_{2}(x).
\end{align*}
Set $ z = u + iv, u, v \in \R^{n} $ and $ \omega^{2n - 1} = \frac{(2\pi)^{n}}{\Gamma(n)} $ be the surface area of $ \mathbb{S}^{2n} $, we analyze $ E_{1}(x) $ and $ E_{2}(x) $ for $ x \in K $. For $ E_{1}(x) $, we use (\ref{REV:eqn1}) and compute that
\begin{align*}
    \sup_{x \in K} E_{1}(x) & \leqslant \sup_{\lvert z \rvert \leq R} \lvert F(z) - \hat{F}(z) \rvert \cdot \left( \frac{2a}{\pi} \right)^{\frac{n}{4}} \left( \frac{a}{\pi} \right)^{n} \sup_{x \in K} \int_{\lvert u \rvert^{2} + \lvert v \rvert^{2} \leq R^{2}} e^{-a \lvert u - x \rvert^{2}} e^{-\frac{a}{2} \lvert u \rvert^{2}} e^{-\frac{a}{2}\lvert v \rvert^{2}} dudv \\
    & \leqslant \sup_{\lvert z \rvert \leq R} \lvert F(z) - \hat{F}(z) \rvert \cdot \left( \frac{2a}{\pi} \right)^{\frac{n}{4}} \left( \frac{a}{\pi} \right)^{n} \omega^{2n - 1} \int_{0}^{R} e^{-\frac{at^{2}}{2}} t^{2n - 1} dt \\
    & = \sup_{\lvert z \rvert \leq R} \lvert F(z) - \hat{F}(z) \rvert \cdot \left( \frac{2a}{\pi} \right)^{\frac{n}{4}} \left( \frac{a}{\pi} \right)^{n} \frac{(2\pi)^{n}}{a^{n}\Gamma(n)} \cdot \gamma\left(n, \frac{aR^{2}}{2} \right) \\
    & \leqslant \sup_{\lvert z \rvert \leq R} \lvert F(z) - \hat{F}(z) \rvert \cdot \left( \frac{2a}{\pi} \right)^{\frac{n}{4}} \cdot 2^{n} \cdot \frac{(aR^{2})^{n}}{2^{n} n!} \\
    & \leqslant e^{\frac{aR^{2}}{2}} \cdot \left( \frac{2a}{\pi} \right)^{\frac{n}{4}} \cdot \frac{(aR^{2})^{n}}{n!} \cdot \lVert \mathbf{c} - \hat{\mathbf{c}} \rVert
\end{align*}
We now control the tail term $ E_{2}(x) $. For any constant $ s \in (0, 1) $, we have
\begin{align*}
    \sup_{x \in K} E_{2}(x) & = \left( \frac{2a}{\pi} \right)^{\frac{n}{4}} \left( \frac{a}{\pi} \right)^{n} \sup_{x \in K} \int_{\lvert z \rvert > R} \left( \lvert e^{-\frac{a}{2} \bar{z} \cdot \bar{z} -a \lvert x \rvert^{2} + 2a \bar{z} \cdot x } \rvert e^{-\frac{(s +1)a}{2}\lvert z \rvert^{2}} \right)\left( \lvert F - \hat{F} \rvert e^{-\frac{(-s +1)a}{2}\lvert z \rvert^{2}} \right) d\mu(z) \\
    & \leqslant \left( \frac{2a}{\pi} \right)^{\frac{n}{4}}  \left( \left( \frac{a}{\pi} \right)^{n} \sup_{x \in K} \int_{\lvert z \rvert > R} \lvert e^{-a \bar{z} \cdot \bar{z} -2a \lvert x \rvert^{2} + 4a \bar{z} \cdot x } \rvert e^{-(s +1)a\lvert z \rvert^{2}} d\mu(z) \right)^{\frac{1}{2}} \\
    & \qquad \times \left( \left( \frac{a}{\pi} \right)^{n}\int_{\lvert z \rvert > R} \lvert F - \hat{F} \rvert^{2} e^{-(-s +1)a\lvert z \rvert^{2}}d\mu(z) \right)^{\frac{1}{2}} \\
    & = \left( \frac{2a}{\pi} \right)^{\frac{n}{4}}  \sup_{x \in K} E_{2, 1}(x) \cdot E_{2, 2}.
\end{align*}
Set $ B_{M} = \sup_{\beta \in \Lambda_{M}} \lvert \beta \rvert $. For $ E_{2, 2} $, we recall {\it{Assumption A}} and the orthonormality of $ \{ e_{\beta}^{a} \} $ (which still holds on any radial region), which implies that
\begin{align*}
    E_{2, 2}^{2} & = \left( \frac{a}{\pi} \right)^{n}\int_{\lvert z \rvert > R} \lvert F - \hat{F} \rvert^{2} e^{-(-s +1)a\lvert z \rvert^{2}}d\mu(z) \\
    & = \sum_{\beta \in \Lambda_{M}} \lvert \hat{c}_{\beta} - c_{\beta} \rvert^{2} \int_{\lvert z \rvert > R} e_{\beta}^{a}(z) e_{\beta}^{a}(\bar{z}) e^{sa \lvert z \rvert^{2}}d \mu_{a}(z) \\
    & \leqslant \sum_{\beta \in \Lambda_{M}} \lvert \hat{c}_{\beta} - c_{\beta} \rvert^{2}  \left( \frac{a}{\pi} \right)^{n} \frac{a^{\lvert \beta \rvert}}{\beta!} \int_{\mathbb{S}^{2n - 1}} \lvert \omega^{\beta} \rvert^{2} d\sigma(\omega) \int_{R}^{\infty} t^{2n + 2\lvert \beta \rvert - 1} e^{-(1 - s)a t^{2}}dt \\
    & = \sum_{\beta \in \Lambda_{M}} \lvert \hat{c}_{\beta} - c_{\beta} \rvert^{2} (1 - s)^{-(n + \lvert \beta \rvert)} \cdot \frac{\Gamma\left(n + \lvert \beta \rvert, (1 - s)aR^{2} \right)}{\Gamma(n + \lvert \beta \rvert)} \\
    & \leqslant (1 - s)^{-(n + B_{M})} \cdot \frac{\Gamma\left(n + B_{M}, (1 - s)aR^{2} \right)}{\Gamma(n + B_{M})} \sum_{\beta \in \Lambda_{M}} \lvert \hat{c}_{\beta} - c_{\beta} \rvert^{2} \\
    & = (1 - s)^{-(n + B_{M})} \cdot \frac{\Gamma\left(n + B_{M}, (1 - s)aR^{2} \right)}{\Gamma(n + B_{M})} \lVert F - \hat{F} \rVert_{\mathcal{F}_{a}^{2}(\C^{n})}^{2}.
\end{align*}
Using the same polar coordinates approach,
\begin{align*}
    \sup_{x \in K} E_{2, 1}(x)^{2} & \leq \left( \frac{a}{\pi} \right)^{n} \int_{\lvert u \rvert^{2} + \lvert v \rvert^{2} > R^{2}}  e^{-as(\lvert u \vert^{2} + \lvert v \rvert^{2})} \left(\sup_{x \in K} e^{-2a\lvert u - x \rvert^{2}} \right) dudv \\
    & \leqslant \left( \frac{a}{\pi} \right)^{n} \cdot e^{-2a(R - r)^{2}} \cdot \omega^{2n - 1} \int_{R}^{\infty} e^{-ast^{2}}t^{2n - 1} dt \\
    & = e^{-2a(R - r)^{2}} s^{-n} \frac{\Gamma(n, saR^{2})}{\Gamma(n)}.
\end{align*}
It follows that
\begin{align*}
\sup_{x \in K} E_{2}(x) & \leqslant \left( \frac{2a}{\pi} \right)^{\frac{n}{4}} \cdot e^{-a(R - r)^{2}} s^{-\frac{n}{2}} (1 - s)^{-\frac{n + B_{M}}{2}} \sqrt{\frac{\Gamma(n, saR^{2})}{\Gamma(n)}\cdot \frac{\Gamma\left(n + B_{M}, (1 - s)aR^{2} \right)}{\Gamma(n + B_{M})}} \lVert F - \hat{F} \rVert_{\mathcal{F}_{a}^{2}(\C^{n})} \\
& \leqslant \left( \frac{2a}{\pi} \right)^{\frac{n}{4}} \cdot e^{-a(R - r)^{2}} s^{-\frac{n}{2}} (1 - s)^{-\frac{n + B_{M}}{2}} \sqrt{\frac{\Gamma(n, saR^{2})}{\Gamma(n)}\cdot \frac{\Gamma\left(n + B_{M}, (1 - s)aR^{2} \right)}{\Gamma(n + B_{M})}}\lVert \mathbf{c} - \hat{\mathbf{c}} \rVert.
\end{align*}
Taking the expectation of the sum of the error bounds of $ E_{1} $ and $ E_{2} $ gives (\ref{REV:eqn4}). We can choose $ a $ small enough, such that the possible dimensional growths in $ n $ are compensated. It follows that (\ref{REV:eqn3}) holds.
\end{proof}
Since {\it{Assumption A}} holds, we can use the Cauchy-Schwarz inequality to get another upper bound of $ \lvert f - \hat{f} \rvert $ on $ \R^{n} $. The following error bound is worse than (\ref{REV:eqn4}) in general bacause of the extra $ \sqrt{M} $ term.
\begin{proposition}\label{REV2:prop1}
Under the same hypotheses of Theorem \ref{REV:thm2}, and the same construction of $ \hat{\mathbf{c}} $, the data-driven approximation $ \hat{f}(x) = \sum_{\beta \in \Lambda_{M}} \hat{c}_{\beta} h_{\beta}^{a}(x) $ satisfies the dimension-free algebraic error rate
\begin{equation}\label{REV2:eqn3}
\mathbb{E}_{Z, X}\big[\sup_{x \in \R^{n}}\lvert f(x) - \hat{f}(x) \rvert \big] = O(L^{-\frac{1}{2}}).
\end{equation}
Precisely speaking,
\begin{equation}\label{REV2:eqn4}
    \mathbb{E}_{X, Z}\big[\sup_{x \in \R^{n}}\lvert f(x) - \hat{f}(x) \rvert \big] \leqslant C_{0}^{n}(2a)^{\frac{n}{4}} \left( 1 + \frac{B_{M}}{n} \right)^{-\frac{n}{12}} \cdot \sqrt{M} \sqrt{\frac{(M - 1)C^{2}}{L} +  \frac{M(C')^{2}}{L}}.
\end{equation}
\end{proposition}
\begin{proof}
By the same complex sampling procedure as in Theorem \ref{CON:thml1}, we apply Cauchy-Schwarz and get
\begin{align*}  \mathbb{E}_{Z, X}\big[\sup_{x \in \R^{n}}\lvert f(x) - \hat{f}(x) \rvert \big] & = \mathbb{E}_{Z, X}\big[\sup_{x \in \R^{n}}\lvert \sum_{\beta \in \Lambda_{M}} (c_{\beta} - \hat{c}_{\beta} ) h_{\beta}^{a} \rvert \big] \leqslant \sup_{x \in \R^{n}, \beta \in \Lambda_{M}} \lvert h_{\beta}^{a} \rvert \cdot \mathbb{E}_{Z, X}\big[ \sum_{\beta \in \Lambda_{M}} \lvert c_{\beta} - \hat{c}_{\beta} \rvert \big] \\
    & \leqslant \sup_{x \in \R^{n}, \beta \in \Lambda_{M}} \lvert h_{\beta}^{a} \rvert \left( \sum_{\beta \in \Lambda_{M}} 1 \right)^{\frac{1}{2}} \mathbb{E}_{\rho_{a}} \left[ \left(\sum_{\beta \in \Lambda_{M}} \lvert c_{\beta} - \hat{c}_{\beta} \rvert^{2} \right)^{\frac{1}{2}} \right] \\
    & \leqslant \sup_{x \in \R^{n}, \beta \in \Lambda_{M}} \lvert h_{\beta}^{a} \rvert \cdot \sqrt{M} \sqrt{\frac{(M - 1)C^{2}}{L} +  \frac{M(C')^{2}}{L}} 
\end{align*}
by the upper bound of $ \mathbb{E}_{Z, X}\big[ \lVert \mathbf{c} - \hat{\mathbf{c}} \rVert_{\C^{M}}^{2} \big]$ in Theorem \ref{CON:thml1}. The canonical upper bound of one-dimensional physicists' Hermitian polynomial says
\begin{equation*}
   \sup_{y \in \R} \lvert h_{m}^{a}(y) \rvert = \sup_{y \in \R} \lvert (2a)^{\frac{1}{4}}H_{m}(\sqrt{2a}y) \rvert \leqslant C_{0} (2a)^{\frac{1}{4}}(m + 1)^{-\frac{1}{12}}.
\end{equation*}
Here $ C_{0} < 1 $ is a constant independent of $ a $ and $ m $. Applying this upper bound and the standard AM-GM inequality to $ h_{\beta}^{a} $ with $ \beta = (\beta_{1}, \dotso, \beta_{n}) $, we get
\begin{align*}
    \sup_{x \in \R^{n}} \lvert h_{\beta}^{a}(x) \rvert & \leqslant C_{0}^{n}(2a)^{\frac{n}{4}} \prod_{j = 1}^{n} (\beta_{j} + 1)^{-\frac{1}{12}} = C^{n}(2a)^{\frac{n}{4}}  \left( \left( \prod_{j = 1}^{n} (\beta_{j} + 1) \right)^{\frac{1}{n}} \right)^{-\frac{n}{12}} \\
    & \leqslant C_{0}^{n}(2a)^{\frac{n}{4}} \left( \frac{1}{n} \sum_{j = 1}^{n}(\beta_{j} + 1) \right)^{-\frac{n}{12}} = C^{n}(2a)^{\frac{n}{4}} \left( 1 + \frac{\lvert \beta \rvert}{n} \right)^{-\frac{n}{12}}.
\end{align*}
Therefore,
\begin{equation*}
    \mathbb{E}_{Z, X}\big[\sup_{x \in \R^{n}}\lvert f(x) - \hat{f}(x) \rvert \big] \leqslant C_{0}^{n}(2a)^{\frac{n}{4}} \left( 1 + \frac{B_{M}}{n} \right)^{-\frac{n}{12}} \cdot \sqrt{M} \sqrt{\frac{(M - 1)C^{2}}{L} +  \frac{M(C')^{2}}{L}},
\end{equation*}
which gives (\ref{REV2:eqn4}). Now we can choose $ a $ small enough to compensate for any potential growth depending on $ n $, which follows that
\begin{equation*}     \mathbb{E}_{\rho_{a}}\big[\sup_{x \in \R^{n}}\lvert f(x) - \hat{f}(x) \rvert \big]  = O\left( L^{-\frac{1}{2}} \right),
\end{equation*}
which is exactly the dimension-free error bound in (\ref{REV2:eqn3}).
\end{proof}
Analogous to Remark \ref{REV:re1}, we achieve the canonical Monte-Carlo rate for the pointwise error estimate between $ f $ and $ \hat{f} $, which is again independent of the dimension. The upper bound in (\ref{REV2:eqn4}) may not be sharp once we restrict $ x \in K $ for some compact subset $ K $. (\ref{REV2:eqn4}) just holds uniformly on $ \R^{n} $.

{\it{Assumption B}} is essential. We give an error bound between $ f $ and $ \hat{f} $ below when we drop the {\it{Assumption A}}, i.e. $ f = \sum_{\beta} c_{\beta} h_{\beta}^{a} $ as an infinite series. Merely with {\it{Assumption B}} and $ f \in \mathcal{H}_{a}^{s}(\R^{n}; \R) $ when $ s < \frac{n}{2} $, we usually have no $ L^{2} $-control between $ f $ and $ \hat{f} $ whose coefficients are estimated from data in original space only. 

Recall the RKHS kernel $ K_{a}(z, w) = e^{a z \cdot \bar{w}} $. Denote the normalized RKHS kernel with complex variable $ z $ and complex parameter $ w $ by
\begin{equation*}
    k_{a}(z, w) = e^{-\frac{a}{2} \lvert w \rvert^{2}}K_{a}(z, w).
\end{equation*}
By the Seip's sampling theorem for Bargmann-Fock space \cite{Seip}, there exists a deterministic complex sampling set $\Lambda = \{ z_{l} \}_{l = 1}^{\infty} $ such that $ \{ k_{a}(z, z_{l}) \}_{l = 1}^{\infty} $ is a frame for $ \mathcal{F}_{a}^{2}(\C^{n}) $ with respect to two positive constants $ A, B $. Define the frame operator
\begin{equation*}
  SF(z) = \sum_{l} \langle F, k_{a}(\cdot, z_{l}) \rangle_{\mathcal{F}_{a}^{2}} k_{a}(z, z_{l})
\end{equation*}
By the Seip's sampling theorem, 
\begin{equation*}
AI \leqslant S \leqslant BI.
\end{equation*}
Hence $ S^{-1} $ is well-defined. We define the normalized dual RKHS kernel
\begin{equation*}
    \tilde{k}_{a}(z, z_{l}) = S^{-1} k_{a}(z, z_{l}) \Rightarrow \tilde{K}_{a}(z, z_{l}) = e^{\frac{a}{2} \lvert z_{l} \rvert^{2}} \tilde{k}_{a}(z, z_{l}).
\end{equation*}
We conclude that
\begin{equation}\label{REV2:eqn5}
F(z) = \sum_{l} \langle F, k_{a}(\cdot, z_{l})\rangle_{\mathcal{F}_{a}^{2}}\tilde{k}_{a}(z, z_{l}) = \sum_{l} \langle F, K_{a}(\cdot, z_{l})\rangle_{\mathcal{F}_{a}^{2}}\tilde{K}_{a}(z, z_{l}) = \sum_{l} F(z_{l}) \tilde{K}_{a}(z, z_{l}).
\end{equation}
Our next result shows that with the Segal-Bargmann transform and the special Seip-type complex grids, we can still get algebraic convergence rate even when $ s < \frac{d}{2} $.
\begin{theorem}\label{REV2:thm3}
Let $ f \in L^{2}(\R^{n}; \R) \cap C^{\infty}(\R^{n}; \R) \cap \mathcal{H}_{s}^{a} $ such that $ \lVert f \rVert_{\mathcal{H}_{a}^{s}(\R^{n};\R)} \leqslant C_{1} $. If only {\it{Assumption B}} holds for some complex random sampling $ \{ V_{l} \}_{l =1}^{L} \overset{i.i.d.}{\sim} \rho_{V} $ with $ \mu_{a} \ll \rho_{V} $, there exists a data-driven approximation $ \hat{f} = \sum_{l = 1}^{L} \hat{d}_{l} \mathcal{B}_{a}^{-1}(\tilde{K}_{a}(\cdot, z_{l}))(x) $ such that
\begin{equation}\label{REV:eqn5}
\mathbb{E}_{V, X}\big[\lVert f - \hat{f} \rVert_{L^{2}(\R^{n}; \R)} \big] \leqslant D_{a, s, \Lambda} L^{-\frac{s}{2n}} \lVert f \rVert_{\mathcal{H}_{a}^{s}(\R^{n};\R)} + D_{a, s, \Lambda}' L^{-\frac{1}{2}} \lVert f \rVert_{L^{2}(\R^{n};\R)},
\end{equation}
where $ \hat{\mathbf{d}} = (\hat{d}_{1}, \dotso, \hat{d}_{L}) $ is determined by random sampling of $ \{ V_{l} \}_{l =1}^{L} \overset{i.i.d.}{\sim} \rho_{V} $ and $ \{ f(X_{i}) \}_{i = 1}^{N} $. In particular, we get the following algebraic error bound 
\begin{equation}\label{REV:eqn6}
\mathbb{E}_{V, X}\big[\lVert f - \hat{f} \rVert_{L^{2}(\R^{n}; \R)} \big] =   O\left( L^{-\frac{s}{2n}} \right).
\end{equation}
\end{theorem}
\begin{proof}
Let $ \tilde{F}_{N}(V_{l}) $ be given as in \ref{REV2:eqn00} such that
\begin{equation*}
    \eta_{l} = F(V_{l}) - \tilde{F}_{N}(V_{l})
\end{equation*}
satisfies (\ref{REV2:eqn0}). Set $ q_{V} = \frac{d\mu_{a}}{d\rho_{V}} $ and define 
\begin{equation}\label{REV2:eqn6}
\hat{d}_{l} = \frac{1}{L}\sum_{j = 1}^{L} \tilde{F}_{N}(V_{j}q_{V}(V_{j}) \overline{K_{a}(V_{j}, z_{l})}, l = 1, \dotso, L.
\end{equation}
Denote by
\begin{equation*}
    P_{L}F(z) = \sum_{l = 1}^{L} \langle F, K_{a}(\cdot, z_{l})\rangle_{\mathcal{F}_{a}^{2}}\tilde{K}_{a}(z, z_{l}),
\end{equation*}
we consider
\begin{equation*}
\mathbb{E}_{V, X} \left[ \lVert F - \hat{F} \rVert_{\mathcal{F}_{a}^{2}(\C^{n})}^{2} \right] \leqslant 2\mathbb{E}_{V, X} \left[ \lVert F - P_{L}F \rVert_{\mathcal{F}_{a}^{2}(\C^{n})}^{2} \right] + 2\mathbb{E}_{V, X} \left[ \lVert P_{L}F - \hat{F} \rVert_{\mathcal{F}_{a}^{2}(\C^{n})}^{2} \right] : = 2E_{1} + 2E_{2}.
\end{equation*}
Since $ f \in \mathcal{H}_{s}^{a}(\R^{n}; \R) $, it follows from the normalized expansion of $F $ in (\ref{REV2:eqn5}) that
\begin{equation*}
\sum_{l} (1 + \lVert z_{l} \rVert^{2})^{s} \lvert \langle F, k_{a}(\cdot, z_{l}) \rangle_{\mathcal{F}_{a}^{2}} \rvert^{2} \leqslant C_{a, s, \Lambda} \lVert f \rVert_{\mathcal{H}_{a}^{s}(\R^{n};\R)}^{2}.
\end{equation*}
In the Seip sampling, we have $ \lVert z_{l} \rVert \leqslant R_{l} $, and we may assume that $ R_{l} \leqslant R_{l + 1}, \forall l $. It follows that there exist constants $ D_{1}, D_{2} $ for a regular sampling lattice, independent of $ L, R_{L} $ such that
\begin{equation*}
    D_{1} R_{L}^{2n} \leqslant L \leqslant D_{2} R_{L}^{2n}.
\end{equation*}
We observe that
\begin{align*}
E_{1} & \leqslant \frac{1}{A}\mathbb{E}_{V, X} \left[ \sum_{l > L} \lvert \langle F, k_{a}(\cdot, z_{l}) \rangle_{\mathcal{F}_{a}^{2}} \rvert^{2}  \right] \leqslant \frac{1}{A} ( 1 + R_{L}^{2})^{-s} \sum_{l}  (1 + \lVert z_{l} \rVert^{2})^{s} \lvert \langle F, k_{a}(\cdot, z_{l}) \rangle_{\mathcal{F}_{a}^{2}} \rvert^{2} \\
& \leqslant \frac{C_{a, s, \Lambda}}{A} ( 1 + R_{L}^{2})^{-s} \lVert f \rVert_{\mathcal{H}_{a}^{s}(\R^{n};\R)}^{2} \leqslant \frac{C_{a, s, \Lambda}}{A} (1 + D_{2}^{-\frac{1}{n}} L^{\frac{1}{n}})^{-s} \lVert f \rVert_{\mathcal{H}_{a}^{s}(\R^{n};\R)}^{2} \\
& : = D_{a, s, \Lambda}^{2} L^{-\frac{s}{n}} \lVert f \rVert_{\mathcal{H}_{a}^{s}(\R^{n};\R)}^{2}.
\end{align*}
For $ E_{2} $, we compute that
\begin{align*}
    E_{2} & \leqslant \mathbb{E}_{V, X}\left[\left\lVert \sum_{l = 1}^{L} \left( F(z_{l}) - \hat{d}_{l} \right) \tilde{K}_{a}(\cdot, z_{l}) \right\rVert_{\mathcal{F}_{a}^{2}(\C^{n})}^{2} \right] \\
    & \leqslant \mathbb{E}_{V, X}\left[\left\lVert \sum_{l = 1}^{L} \left( F(z_{l}) - \frac{1}{L} \sum_{j = 1}^{L} F(V_{j})q_{V}(V_{j}) \overline{K_{a}(V_{j}, z_{l})}  \right) \tilde{K}_{a}(\cdot, z_{l}) \right\rVert_{\mathcal{F}_{a}^{2}(\C^{n})}^{2} \right] \\
    & \qquad + \mathbb{E}_{V, X}\left[\left\lVert \sum_{l = 1}^{L} \left( \frac{1}{L} \sum_{j = 1}^{L} F(V_{j})q_{V}(V_{j}) \overline{K_{a}(V_{j}, z_{l})} - \hat{d}_{l} \right) \tilde{K}_{a}(\cdot, z_{l}) \right\rVert_{\mathcal{F}_{a}^{2}(\C^{n})}^{2} \right] \\
    & = \mathbb{E}_{V}\left[\mathbb{E}_{X|V} \left[ \left\lVert  \sum_{l = 1}^{L} \mathbb{E}_{X|V} \left( F(z_{l}) - \frac{1}{L} \sum_{j = 1}^{L} F(V_{j})q_{V}(V_{j}) \overline{K_{a}(V_{j}, z_{l})} \right) \tilde{K}_{a}(\cdot, z_{l}) \right\rVert_{\mathcal{F}_{a}^{2}(\C^{n})}^{2} \right] \right] \\
    & \qquad + \mathbb{E}_{V, X}\left[\mathbb{E}_{X|V} \left[\left\lVert \sum_{l = 1}^{L}  \left( \frac{1}{L} \sum_{j = 1}^{L} \eta_{j}q_{V}(V_{j}) \overline{K_{a}(V_{j}, z_{l})} \right) \tilde{K}_{a}(\cdot, z_{l}) \right\rVert_{\mathcal{F}_{a}^{2}(\C^{n})}^{2} \right] \right] 
\end{align*}
By the same argument of Monte-Carlo sampling with respect to $ \{ V_{l} \}_{l = 1}^{L} \overset{i.i.d.}{\sim} \rho_{V} $ as in Theorem \ref{CON:thml1}, we conclude from (\ref{REV2:eqn0}) that
\begin{equation*}
E_{2} \leqslant (D_{a, s, \Lambda}')^{2}L^{-1} \lVert f \rVert_{L^{2}(\R^{n}; \R)}^{2}
\end{equation*}
where $ D_{a, s, \Lambda}' $ depends on the upper bound of the Fock norms of both $ K_{a} $ and $ \tilde{K}_{a} $.

Note that by definition of $ \hat{f} $, we have
\begin{equation*}
    \mathcal{B}_{a}^{-1}(F - \hat{F}) = f - \hat{f}.
\end{equation*}
By the unitary property of the inverse Segal-Bargmann transform, we get (\ref{REV2:eqn5}) by adding estimates of $ E_{1} $ and $ E_{2} $ together, after absorbing fixed numerical factors into the constants.
\end{proof}
\medskip

\subsection{$ L^{\infty} $-Convergence of Solving Partial Differential Equations by Neural Operators}
Consider $ \mathcal{L}u  = f $ on $ \R^{n} $ where $ \mathcal{L} $ is a smooth differential operator. Our goal is to learn $ \mathcal{L}^{-1} : f \rightarrow u $. In this section, we show that we can apply complex neural operator $ \tilde{G}_{\tilde{\theta}} $ to approximate $ \mathcal{L}^{-1} $, and give the error bounds. The first version requires both {\it{Assumption A}} and {\it{Assumption B}}, while the second version requires merely {\it{Assumption B}}. In addition, we give an algorithm for the {\it{information recovery matrix}} $ W $.

The next lemma verifies the diagram (\ref{AL:eqn1a}).
\begin{lemma}\label{CON:lemma1}
Let $ \mathcal{L} $ be a differential operator with smooth coefficients. For $ u, f \in L^{2}(\R^{n}; \R) \cap \calC^{\infty}(\R^{n}; \R) $. Suppose that there exists a continuous linear functional $ \tilde{G}  $ such that $ \mathcal{B}_{a}u = \tilde{G}\mathcal{B}_{a} f $, then the following relation holds.
\begin{equation*}
\mathcal{L}^{-1}f = \mathcal{B}_{a}^{-1} \tilde{G} \mathcal{B}_{a}f
\end{equation*}
for every $ u $ such that $ \mathcal{L}u = f $.
\end{lemma}
\begin{proof}
The proof is immediate by the linearity of $ \mathcal{B}_{a} $ and $ \mathcal{B}_{a}^{-1} $.
\end{proof}

Let's discuss the simple case where Lemma \ref{CON:lemma1} holds. We also need the universal approximation theorem for neural operators on Hilbert space:
\begin{proposition}\label{REV2:prop2}\cite{CHEN}
Let $ \tilde{G} : \mathcal{F}_{a}^{2}(\C^{n}) \rightarrow \mathcal{F}_{a}^{2}(\C^{n}) $ be a solution operator of the following equation
\begin{equation*}
    U = \tilde{G} F
\end{equation*}
on the Hilbert space $ \mathcal{F}_{a}^{2}(\C^{n}) $, which satisfies
\begin{equation}\label{CON:eqn6}
\lVert \tilde{G}(F_{1}) - \tilde{G}(F_{2}) \rVert_{\mathcal{F}_{a}^{2}(\C^{n})} \leqslant C_{\tilde{G}, \mathcal{K_{1}}} \lVert F_{1} - F_{2} \rVert_{\mathcal{F}_{a}^{2}(\C^{n})}, \forall F_{1}, F_{2} \in \mathcal{K_{1}} \subset \mathcal{F}_{a}^{2}(\C^{n})
\end{equation}
for some compact admissible set $ \mathcal{K}_{1} $. Then for any $ \epsilon_{NO} > 0 $, there exists a neural operator $ \tilde{\mathcal{G}}_{\tilde{\theta}} $  parametrized by $ \tilde{\theta} $ such that
\begin{equation*}
\sup_{F \in \mathcal{K}_{1}} \left\lVert \tilde{G}(F)(z) - \tilde{\mathcal{G}}_{\tilde{\theta}}(F)(z) \right\rVert_{\mathcal{F}_{a}^{2}(\C^{n})} < \epsilon_{NO}.
\end{equation*}
\end{proposition}
We are now ready to give our first result in neural operator estimate of the operator $ \mathcal{L}^{-1} $. 
\begin{theorem}\label{REV2:thm2}
Let $ s \geqslant s_{0} > \frac{n}{2} + 2 $ and $ u, f \in \mathcal{H}_{a}^{s}(\R^{n}; \R) \cap \calC^{\infty}(\R^{n}; \R) $ satisfying $ \mathcal{L} u = f $, where $ \mathcal{L} $ is a smooth differential operator whose inverse $ \mathcal{L}^{-1} $ is also a smooth operator. Suppose that we only have finitely many observations $ \{ f(x_{i}) \}_{i = 1}^{N} $ of the unknown ground truth $ f $. If {\it{Assumption A}} holds for $ f $, and {\it{Assumption B}} holds for some i.i.d. complex sampling $ \{ Z_{l} \}_{l = 1}^{L} \subset K \subset \C^{n} $, then for any $ \epsilon > 0 $ and a fixed compact admissible set $ \mathcal{K} \subset L^{2}(\R^{n}; \R) $ that contains $ \overline{\text{Span}(V_{M})} $, there exists a neural operator $ \tilde{G}_{NO} $ parametrized by $ \tilde{\theta} $ such that
\begin{equation}\label{REV2:eqn7}
\begin{split}
    & \mathbb{E}_{Z, X}\left[ \sup_{x \in \R^{n}, f \in \mathcal{K}, \{ Z_{l} \} \subset K} \lvert \mathcal{L}^{-1}(f)(x) - \mathcal{B}_{a}^{-1} \circ \tilde{\mathcal{G}}_{NO} \circ \mathcal{B}_{a}(\hat{f})(x) \rvert \right] \\
    & \qquad \leqslant C_{0}^{n}(2a)^{\frac{n}{4}} \left( 1 + \frac{B_{M}}{n} \right)^{-\frac{n}{12}} \cdot \sqrt{M}\left( \epsilon_{NO} + C_{\tilde{G}, \mathcal{K}} \sqrt{\frac{(M - 1)\lVert f \rVert_{L^{2}(\R^{n}; \R)}^{2}}{L} +  \frac{M\lVert f \rVert_{L^{\infty}(\R^{n}; \R)}^{2}}{L}} \right)
\end{split}
\end{equation}
where $ \hat{f} = \sum_{\beta \in \Lambda_{M}} \hat{c}_{\beta}h_{\beta}^{a} $ is determined by $ \{ Z_{l} \} $ and $ \{ f(x_{i}) \} $ as in Theorem \ref{REV:thm2}. In particular, there exists an appropriate choice of $ a $ such that the following dimension-free error bound holds:
\begin{equation}\label{REV2:eqn8}
\mathbb{E}_{Z, X}\left[ \sup_{x \in \R^{n},f \in \mathcal{K}, \{ Z_{l} \} \subset K} \lvert L^{-1}(f) - \mathcal{B}_{a}^{-1} \circ \tilde{\mathcal{G}}_{NO} \circ \mathcal{B}_{a}(\hat{f}) \rvert\right] = O\left( L^{-\frac{1}{2}} \right).
\end{equation}
\end{theorem}
\begin{proof}
Fix some $ \epsilon > 0 $. Since $ \mathcal{B}_{a} $ is a unitary transformation, $ \mathcal{K}' : =\mathcal{B}_{a} (\mathcal{K}) $ is a compact admissible set of $ \mathcal{F}_{a}^{2}(\C^{n}) $. By our hypotheses, $ \mathcal{B}_{a}(\hat{f}), \mathcal{B}_{a}(f) \in \mathcal{K}' $. Fix a positive constant $ \epsilon_{NO} < \epsilon $ that will be determined later. By Lemma \ref{CON:lemma1} and Proposition \ref{REV2:prop2}, there exists a neural operator $ \tilde{\mathcal{G}}_{NO} $ such that
\begin{align*}
    & \mathbb{E}_{Z, X}\left[ \sup_{\mathcal{B}_{a}(f) \in \mathcal{K}', \{ Z_{l} \} \subset K}\left\lVert \tilde{G}(\mathcal{B}_{a}(f)) - \tilde{\mathcal{G}}_{NO} \circ \mathcal{B}_{a}(\hat{f}) \right\rVert_{\mathcal{F}_{a}^{2}(\C^{n})} \right] \\
    & \qquad \leqslant \mathbb{E}_{Z, X}\left[ \left\lVert \tilde{G}(\mathcal{B}_{a}(f)) - \tilde{G} \circ \mathcal{B}_{a}(\hat{f}) \right\rVert_{\mathcal{F}_{a}^{2}(\C^{n})} \right] \\
    & \qquad \qquad + \mathbb{E}_{Z, X}\left[ \sup_{\mathcal{B}_{a}(f) \in \mathcal{K}', \{ Z_{l} \} \subset K}\left\lVert \tilde{G} \circ \mathcal{B}_{a}(\hat{f}) - \tilde{\mathcal{G}}_{NO} \circ \mathcal{B}_{a}(\hat{f}) \right\rVert_{\mathcal{F}_{a}^{2}(\C^{n})} \right] \\
    & \qquad \leqslant C_{\tilde{G}, \mathcal{K}'}\mathbb{E}_{Z, X}\left[ \left\lVert \tilde{G}(\mathcal{B}_{a}(f)) - \tilde{G} \circ \mathcal{B}_{a}(\hat{f}) \right\rVert_{\mathcal{F}_{a}^{2}(\C^{n})} \right] + \epsilon_{NO} \\
    & \qquad \leqslant C_{\tilde{G}, \mathcal{K}'} \sqrt{\frac{(M - 1)C^{2}}{L} +  \frac{M(C')^{2}}{L}} + \epsilon_{NO}.
\end{align*}
Apply Lemma \ref{CON:lemma1} and Proposition \ref{REV2:prop1}, we conclude that
\begin{align*}
     & \mathbb{E}_{Z, X}\left[ \sup_{x \in \R^{n}, f \in \mathcal{K}, \{ Z_{l} \} \subset K} \lvert L^{-1}(f)(x) - \mathcal{B}_{a}^{-1} \circ \tilde{\mathcal{G}}_{NO} \circ \mathcal{B}_{a}(\hat{f})(x) \rvert \right] \\
     & \qquad =  \mathbb{E}_{Z, X}\left[ \sup_{x \in \R^{n}, f \in \mathcal{K}, \{ Z_{l} \} \subset K} \left\lvert \mathcal{B}_{a}^{-1} \left( \tilde{G}(\mathcal{B}_{a}(f)) - \tilde{\mathcal{G}}_{NO} \circ \mathcal{B}_{a}(\hat{f}) \right)(x) \right\rvert \right] \\
     & \qquad \leqslant C_{0}^{n}(2a)^{\frac{n}{4}} \left( 1 + \frac{B_{M}}{n} \right)^{-\frac{n}{12}} \cdot \sqrt{M} \left( C_{\tilde{G}, \mathcal{K}} \sqrt{\frac{(M - 1)C^{2}}{L} +  \frac{M(C')^{2}}{L}} + C_{1} (2a)^{\frac{n}{4}} \epsilon_{NO} \right) \\
     & \qquad \leqslant C_{0}^{n}(2a)^{\frac{n}{4}} \left( 1 + \frac{B_{M}}{n} \right)^{-\frac{n}{12}} \cdot \sqrt{M} \left( \epsilon + C_{\tilde{G}, \mathcal{K}} \sqrt{\frac{(M - 1)C^{2}}{L} +  \frac{M(C')^{2}}{L}} \right).
\end{align*}
Here the bound $ C_{1} (2a)^{\frac{n}{4}} $ comes from the integration $ \int_{\C^{n}} \lvert A_{a}(\bar{z}, x) \rvert d\mu_{a}(z) $, where $ C_{1} $ is independent of the dimension. We can shrink $ \epsilon_{NO} $ if necessary so that the last inequality above holds. This gives (\ref{REV2:eqn7}). With an appropriate choice of $ a $, the $ a^{\frac{n}{4}} $ term is small enough to compensate any potential dimensional growth in $ M $ and $ C_{0}^{n} $. Hence (\ref{REV2:eqn8}) follows.
\end{proof}
\begin{remark}\label{REV2:re3}
If we define
\begin{equation*}
    \mathcal{G}_{\theta} : = \mathcal{B}_{a}^{-1} \circ \tilde{\mathcal{G}}_{NO} \circ \mathcal{B}_{a},
\end{equation*}
we can apply Markov's inequality for fixed $ \delta \in (0, 1) $ to show that the following error bound holds in probability at least $ 1 - \delta $:
\begin{equation}\label{REV2:eqn9}
\sup_{x \in \R^{n}, f \in \mathcal{K}, \{ Z_{l} \} \subset K } \lvert u(x) - \mathcal{G}_{\theta}(\hat{f})(x) \rvert = O\left( L^{-\frac{1}{2}} \right).
\end{equation}
It means that with the help of complex grids $ \{ z_{l} \}_{l = 1}^{L} $, we can start at a finite set of observations $ \{ f(x_{i}) \}_{i = 1}^{N} $ to solve the PDE with $ O\left( L^{-\frac{1}{2}} \right) $ error, provided that {\it{Assumption A}} and {\it{Assumption B}} hold.

If we only know that the observations $ \{ f(x_{i}) \} $ are evaluated on grids $ \{ x_{i} \} $ that are sampled from a collection of random variables $ \{ X_{i} \} $ but nothing else, we cannot reach this conclusion.
\end{remark}
If we assume $ f \in \mathcal{H}_{a}^{s}(\R^{n}; \R) $, we can drop the {\it{Assumption A}} and get the same convergence rate as in (\ref{REV2:eqn8}).
\begin{theorem}\label{REV2:thm4}
Let $ u, f \in L^{2}(\R^{n}; \R) \cap \calC^{\infty}(\R^{n}; \R) $ satisfying $ \mathcal{L} u = f $, where $ L $ is a smooth differential operator whose inverse $ \mathcal{L}^{-1} $ is a bounded smooth operator. Suppose that we only have finitely many observations $ \{ f(x_{i}) \}_{i = 1}^{N} $ of the unknown ground truth $ f $. If only {\bf{Assumption B}} holds for $ P_{M}f $ for some i.i.d. complex sampling $ \{ Z_{l} \}_{l = 1}^{L} \subset K \subset \C^{n} $, then for any $ \epsilon > 0 $ and a fixed compact admissible set $ \mathcal{K} \subset L^{2}(\R^{n}; \R) $ that contains $ \overline{\text{Span}(V_{M})} $, there exists a neural operator $ \tilde{G}_{NO} $ parametrized by $ \tilde{\theta} $  such that
\begin{equation}\label{REV2:eqn9}
\begin{split}
    & \mathbb{E}_{Z, X}\left[ \sup_{x \in \R^{n}, f \in \mathcal{K}, \{ Z_{l} \} \subset K} \lvert \mathcal{L}^{-1}(f)(x) - \mathcal{B}_{a}^{-1} \circ \tilde{\mathcal{G}}_{NO} \circ \mathcal{B}_{a}(\hat{f})(x) \rvert \right] \\
    & \qquad \leqslant C_{0}^{n}(2a)^{\frac{n}{4}} \left( 1 + \frac{B_{M}}{n} \right)^{-\frac{n}{12}} \cdot \sqrt{M} \left( \epsilon + C_{\tilde{G}, \mathcal{K}} \sqrt{\frac{(M - 1)\lVert f \rVert_{L^{2}(\R^{n}; \R)}^{2}}{L} +  \frac{C_{M} M\lVert f \rVert_{L^{\infty}(\R^{n}; \R)}^{2}}{L}} \right) \\
    & \qquad \qquad + (1 + 2a(2M_{0} + n))^{-\frac{s - s_{0}}{2}} (4\pi)^{-\frac{n}{4}} \sqrt{\frac{\Gamma(s_{0} - \frac{n}{2})}{\Gamma(s_{0})}} \lVert \mathcal{L}^{-1} \rVert \lVert f \rVert_{\mathcal{H}_{a}^{s}}
\end{split}
\end{equation}
where $ \hat{f} = \sum_{\beta \in \Lambda_{M}} \hat{c}_{\beta}h_{\beta}^{a} $ determined by $ \{ Z_{l} \} $ and $ \{ f(x_{i}) \} $ as in Theorem \ref{REV:thm2}. In particular, there exists an appropriate choice of $ a $ such that the following dimension-free error bound holds:
\begin{equation}\label{REV2:eqn10}
\mathbb{E}_{Z, X}\left[ \sup_{x \in \R^{n}, f \in \mathcal{K}, \{ Z_{l} \} \subset K} \lvert \mathcal{L}^{-1}(f) - \mathcal{B}_{a}^{-1} \circ \tilde{\mathcal{G}}_{NO} \circ \mathcal{B}_{a}(\hat{f}) \rvert\right] = O\left( L^{-\frac{1}{2}} \right).
\end{equation}
\end{theorem}
\begin{proof}
Clearly,
\begin{align*}
    & \mathbb{E}_{Z, X}\left[ \sup_{x \in \R^{n}, f \in \mathcal{K}, \{ Z_{l} \} \subset K} \lvert \mathcal{L}^{-1}(f)(x) - \mathcal{B}_{a}^{-1} \circ \tilde{\mathcal{G}}_{NO} \circ \mathcal{B}_{a}(\hat{f})(x) \rvert \right] \\
    & \qquad \leqslant \mathbb{E}_{Z, X}\left[ \sup_{x \in \R^{n}, f \in \mathcal{K}, \{ Z_{l} \} \subset K} \lvert \mathcal{L}^{-1}(f)(x) - \mathcal{L}^{-1}(P_{M}f)(x) \rvert \right] \\ 
    & \qquad \qquad \mathbb{E}_{Z, X}\left[ \sup_{x \in \R^{n}, f \in \mathcal{K}, \{ Z_{l} \} \subset K} \lvert \mathcal{L}^{-1}(P_{M}f)(x) - \mathcal{B}_{a}^{-1} \circ \tilde{\mathcal{G}}_{NO} \circ \mathcal{B}_{a}(\hat{f})(x) \rvert \right] \\
    & \qquad  = E_{1} + E_{2}.
\end{align*}
For $ E_{1} $, we apply (\ref{CON:eqnl1}) in Lemma \ref{CON:lemma0},
\begin{align*}
    E_{1} & = \sup_{x \in \R^{n}, f \in \mathcal{K}, \{ Z_{l} \} \subset K} \lvert \mathcal{L}^{-1}(f)(x) - \mathcal{L}^{-1}(P_{M}f) \rvert \\
    & \leqslant  (1 + 2a(2M_{0} + n))^{-\frac{s - s_{0}}{2}} (4\pi)^{-\frac{n}{4}} \sqrt{\frac{\Gamma(s_{0} - \frac{n}{2})}{\Gamma(s_{0})}} \left\lVert \mathcal{L}^{-1}(f) - \mathcal{L}^{-1}(P_{M}f) \right\rVert_{\mathcal{H}_{a}^{s}(\R^{n}; \R)} \\
    & \leqslant (1 + 2a(2M_{0} + n))^{-\frac{s - s_{0}}{2}} (4\pi)^{-\frac{n}{4}} \sqrt{\frac{\Gamma(s_{0} - \frac{n}{2})}{\Gamma(s_{0})}} \lVert \mathcal{L}^{-1} \rVert \left\lVert f \right\rVert_{\mathcal{H}_{a}^{s}(\R^{n}; \R)}
\end{align*}
By the hypotheses of the statement with respect to $ P_{M} f $, we apply (\ref{REV2:eqn7}) and conclude that $ E_{2} $ satisfies
\begin{align*}
    & \mathbb{E}_{Z, X}\left[ \sup_{x \in \R^{n}, f \in \mathcal{K}, \{ Z_{l} \} \subset K} \lvert \mathcal{L}^{-1}(P_{M}f)(x) - \mathcal{B}_{a}^{-1} \circ \tilde{\mathcal{G}}_{NO} \circ \mathcal{B}_{a}(\hat{f})(x) \rvert \right] \\
    & \qquad \leqslant C_{0}^{n}(2a)^{\frac{n}{4}} \left( 1 + \frac{B_{M}}{n} \right)^{-\frac{n}{12}} \cdot \sqrt{M} \left( \epsilon + C_{\tilde{G}, \mathcal{K}} \sqrt{\frac{(M - 1)\lVert P_{M}f \rVert_{L^{2}(\R^{n}; \R)}^{2}}{L} +  \frac{M\lVert P_{M}f \rVert_{L^{\infty}(\R^{n}; \R)}^{2}}{L}} \right) \\
    & \qquad \leqslant C_{0}^{n}(2a)^{\frac{n}{4}} \left( 1 + \frac{B_{M}}{n} \right)^{-\frac{n}{12}} \cdot \sqrt{M} \left( \epsilon + C_{\tilde{G}, \mathcal{K}} \sqrt{\frac{(M - 1)\lVert f \rVert_{L^{2}(\R^{n}; \R)}^{2}}{L} +  \frac{C_{M} M\lVert f \rVert_{L^{\infty}(\R^{n}; \R)}^{2}}{L}} \right)
\end{align*}
where $ C_{M} $ depends only on $ f $ and $ M $. Combing the upper bound of $ E_{1} $ and $ E_{2} $ gives (\ref{REV2:eqn9}). By the same argument of Theorem \ref{REV2:thm2}, there exists an appropriate $ a \ll 1 $ such that (\ref{REV2:eqn10}) holds, which is dimension-free.
\end{proof}
\medskip

\section{Numerical Experiments}
In this section, we study the Segal--Bargmann neural operator (SBNO) through
2D Darcy flow, the 1D Burgers' equation, 
and two nonlinear
elliptic equations in spatial dimension $ 5 $. Based on our best understanding, the state-of-the-art neural operator approach for unknown $ \mathcal{L} $ is spatial dimension $ 3 $. Our goal is to learn the corresponding
solution operators from finitely many observations of input and output functions.
For the first three PDEs, we compare SBNO with FNO on a regular grid, and with GNOT and BelNet on scattered
points. Within each comparison, the methods use the same training and testing
functions and the same observations. The Darcy and Burgers'
experiments use $2000$ training functions and $200$ independent testing functions.
For the scattered-point experiments, we also evaluate the same testing functions
on a new point set to study how the learned operators depend on the observation locations.

For the Burgers' 
equation, the models receive only
observations of the initial field and predict solution values throughout
$0\leq t\leq1$ in one evaluation. Different from standard benchmark comparisons, no intermediate solution snapshots are given
as inputs for our experiments. This requires learning the evolution over the whole interval
without recursive time stepping. The new-point test is more demanding than
testing only at the training locations, since the operator must generalize
to both unseen initial conditions and unseen space--time sampling points.
All tests remain within the training time interval.

\subsection{Two-dimensional Darcy flow}
We first consider the Darcy equation
\begin{equation}\label{exp:darcy}
    -\nabla\cdot\bigl(\kappa(x)\nabla u(x)\bigr)=g
    \quad {\rm in} \Omega,\qquad
    u=0\quad\text{on }\partial \Omega,
\end{equation}
where $\kappa$ is the positive permeability field. The input is the collection
$\{\kappa(x_i)\}_{i=1}^{N}$, and the output consists of the solution values
$\{u(x_i)\}_{i=1}^{N}$. Thus, we seek an approximation of the map $\kappa\mapsto u$
with the forcing and boundary conditions fixed.

For the comparison with FNO~\cite{FNO}, we take $\Omega=(-1,1)^{2}$ and $g=1$.
The permeability fields are drawn from four families: discontinuous mosaics,
rotated mosaics, curved channels, and continuous multiscale fields with sharp
transitions. Their random parameters are chosen independently for each function,
with $0.1\leq\kappa\leq10$.
We solve~\eqref{exp:darcy} on a $1024\times1024$ uniform grid by a conservative
five-point finite-difference scheme, using harmonic averages of permeability
across adjacent nodes. Retaining every sixteenth node along each coordinate
gives the $64\times64$ grid used by both SBNO and FNO.

For the comparison with GNOT~\cite{GNOT} and BelNet~\cite{BelNet}, we use
$\Omega=(-1,1)^2$ and $g=-1$, which gives the convention
$\nabla\cdot(\kappa\nabla u)=1$ in the scattered-point dataset.
Here, $\kappa$ is a random expansion in $256$ normalized Hermite functions,
with a positive ground-state component and bounded perturbations chosen to
preserve positivity. We draw $1024^2$ interior nodes from a density proportional
to $e^{-2|x|^2}$ on $D$, add boundary nodes, and form a Delaunay triangulation.
The reference solutions are computed by conforming piecewise-linear finite
elements. We then select $32^2=1024$ interior nodes without replacement and take
the permeability and solution values directly from those nodes. In this
experiment, $32\times32$ denotes the number of observations; the points are
scattered rather than arranged as a tensor grid.

We plot the relative $ L^{2} $ test error of both SBNO and FNO on original $ 64 \times 64 $ uniform grids and expanded $ 128 \times 128 $ grids, the $ \mathbb{E}_{X}\left[ \frac{\lvert \hat{u}(X) - u(X) \rvert}{\lvert u(X) \rvert} \right]$ error, and the relative $ L^{\infty} $ error in Figure 1. Figure 1 demonstrates that SBNO outperforms FNO for all three types of errors, for both original grids and for generalization; in particular, the typical relative solution error for SBNO is one order better than FNO. 

We point out that FNO has a lower expected relative error on the  on the $ 128 \times 128 $ grid than the $ 64 \times 64 $ grid. This is a reasonable result since the smaller proportion of grid points in the outermost grid layers are chosen for $ 128 \times 128 $ grids, where the errors are larger for FNO. In contrast, the numerical results illustrates the stability of our SBNO approach everywhere.

We also give the training speed comparison table between a fast algorithm of SBNO and FNO in Table 1 for a $ 32 \times 32 $ original grid test. Note that the fast SBNO algorithm is not used in training the results in Figure 1 above. Table 1 shows that the fast SBNO can reach the same accuracy as FNO, meanwhile 10 times faster than the FNO. 

\begin{figure}[h]
\includegraphics[width=165mm, height=120mm]{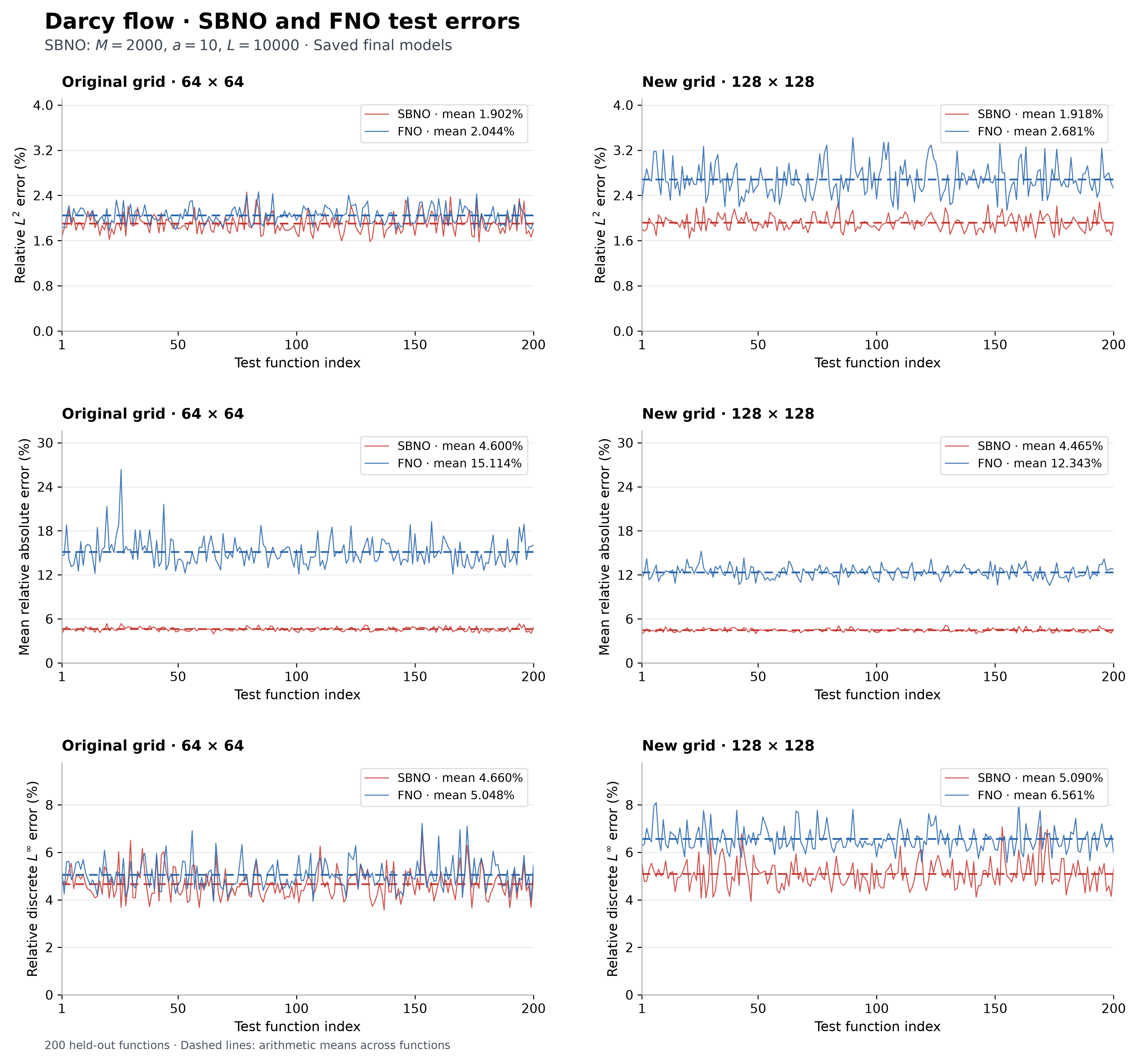} 
    \caption{2D Darcy Flow FNO vs SBNO}
\end{figure}
\begin{table}[ht]
    \centering
    \renewcommand{\arraystretch}{1.35}
    \setlength{\tabcolsep}{14pt}

    \begin{tabular}{lcc}
        \toprule

        \multicolumn{3}{c}{\Large\bfseries Benchmark Comparison: SBNO vs.\ FNO} \\

        \midrule

        \textbf{Metric}
        &
        \textbf{SBNO}
        &
        \textbf{FNO}
        \\
        \midrule

        Training-pipeline time (s)
        &
        $T_{\mathrm{train}}^{\mathrm{SBNO}} = 43.43s$
        &
        $T_{\mathrm{train}}^{\mathrm{FNO}} = 372.67s$
        \\

        Wall-clock time (s)
        &
        $T_{\mathrm{wall}}^{\mathrm{SBNO}} = 45.70s$
        &
        $T_{\mathrm{wall}}^{\mathrm{FNO}} = 372.68s$
        \\

       Training squared $ L^{2} $ loss & $E_{\mathrm{train}}^{\mathrm{SBNO}} = 3.85 \times 10^{-4} $ & $E_{\mathrm{train}}^{\mathrm{FNO}} = 4.86 \times 10^{-4}$ \\ 
       
       Relative Test $ L^{2} $ error & $E_{\mathrm{test}}^{\mathrm{SBNO}}= 1.94 \times 10^{-2}$ & $E_{\mathrm{test}}^{\mathrm{FNO}} = 2.03 \times 10^{-2} $ \\
       
        Optimizer steps
        &
        $N_{\mathrm{step}}^{\mathrm{SBNO}} = 5000 $
        &
        $N_{\mathrm{step}}^{\mathrm{FNO}} = 50000 $
        \\

        \bottomrule
    \end{tabular}

    \caption{Computational benchmark comparison between the
    Segal--Bargmann Neural Operator (SBNO) and the Fourier
    Neural Operator (FNO).}

    \label{tab:sbno-fno-benchmark}

\end{table}

On a total of $ 1024 $ scattered grids, we also plot the relative $ L^{2} $ test error of SBNO, GNOT and BelNet, the $ \mathbb{E}_{X}\left[ \frac{\lvert \hat{u}(X) - u(X) \rvert}{\lvert u(X) \rvert} \right]$ error, and the relative $ L^{\infty} $ error in Figure 2. Figure 2 demonstrates that SBNO is comparable with GNOT and outperforms BelNet for all three types of errors for unseen test functions on original grids; in particular, the testing error of SBNO is one order better than the BelNet. Although the GNOT performs a little better in this scenario, this experiments also demonstrates the wide applicability of our SBNO approach, especially we give experimental evidence of the outstanding behavior of SBNO for other PDEs.

\begin{figure}[h]
\includegraphics[width=165mm, height=120mm]{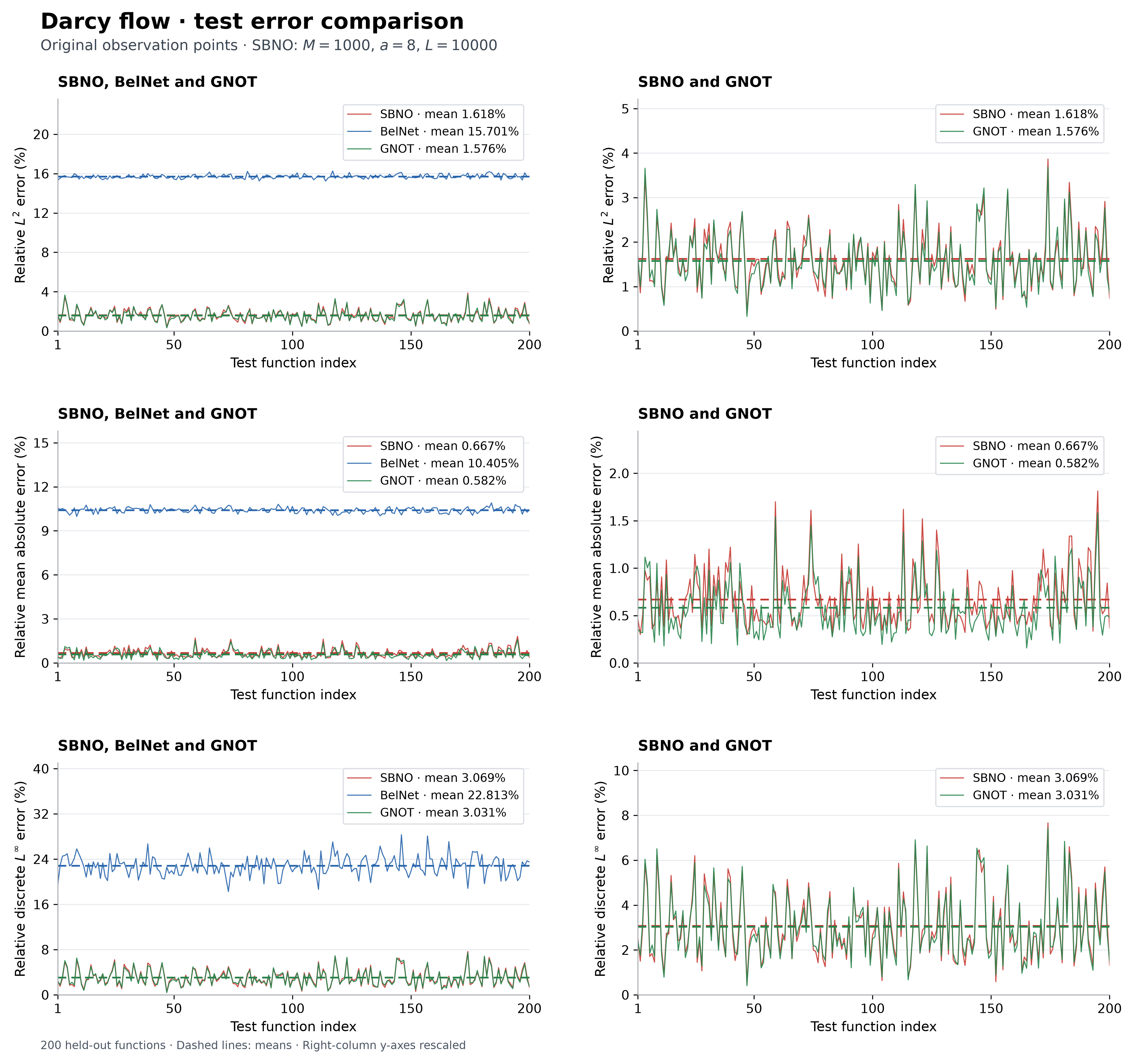} 
    \caption{2D Darcy Flow SBNO vs GNOT vs BelNet}
\end{figure}

\subsection{The one-dimensional Burgers' equation}
We next consider the viscous Burgers' equation on $[0,1]\times(-1,1)$:
\begin{equation}\label{exp:burgers}
    \partial_t u+u\,\partial_xu=0.1\,\partial_{xx}u,
    \qquad u(0,x)=u_0(x),
\end{equation}
with periodic boundary conditions
$u(t,-1)=u(t,1)$ and $\partial_xu(t,-1)=\partial_xu(t,1)$.
Our goal is to learn the map $u_0\mapsto u(t,x)$ over the whole time interval.
At a sampled point $(t_i,x_i)$, the input value is $u_0(x_i)$ and the output
value is $u(t_i,x_i)$. The initial condition is therefore extended constantly
in time when constructing the input observations.

We generate $u_0$ as a zero-mean Gaussian Fourier series with frequencies
$1,\ldots,16$, using independent Gaussian sine and cosine coefficients whose
variances decrease with frequency. A fixed scaling gives an expected spatial
mean square of $0.5^2$.
The reference solution is obtained from the Cole--Hopf formula, using $2048$
quadrature points for the initial heat potential and $128$ positive Fourier
frequencies for its evolution. This representation allows direct evaluation at
the scattered space--time points.
Writing $s_t=2t-1$, we draw $1024^2$ points $(s_t,x)$ from the same truncated
Gaussian law as in the scattered Darcy experiment and select $32^2=1024$ of
them for training and testing. SBNO is compared with GNOT~\cite{GNOT} and
BelNet~\cite{BelNet} using these observations.

For each test initial condition, predictions are compared with the reference
values at the same $1024$ space--time points. We first use the training point
set and then a new, independently sampled set
$\{(t'_i,x'_i)\}_{i=1}^{1024}$, with inputs $u_0(x'_i)$ and targets
$u(t'_i,x'_i)$. The same $200$ unseen initial conditions are used for both
tests. Thus, $32^2$ counts scattered points over the whole interval; it does
not denote $32$ time slices. The second test additionally measures
generalization to new observation locations.

On a total of $ 1024 $ scattered grids in $ [0, 1]_{t} \times [-1, 1]_{x} $, we plot the relative $ L^{2} $ test error of SBNO, GNOT and BelNet, the $ \mathbb{E}_{X}\left[ \frac{\lvert \hat{u}(X) - u(X) \rvert}{\lvert u(X) \rvert} \right]$ error, and the relative $ L^{\infty} $ error in Figure 3. Figure 3 demonstrates that SBNO outperforms GNOT and BelNet for all three types of errors for unseen test functions on original grids; in particular, the testing error of SBNO is at least half of the GNOT
for all three types of errors, and is one order better than the BelNet.
\begin{figure}[h]
\includegraphics[width=165mm, height=70mm]{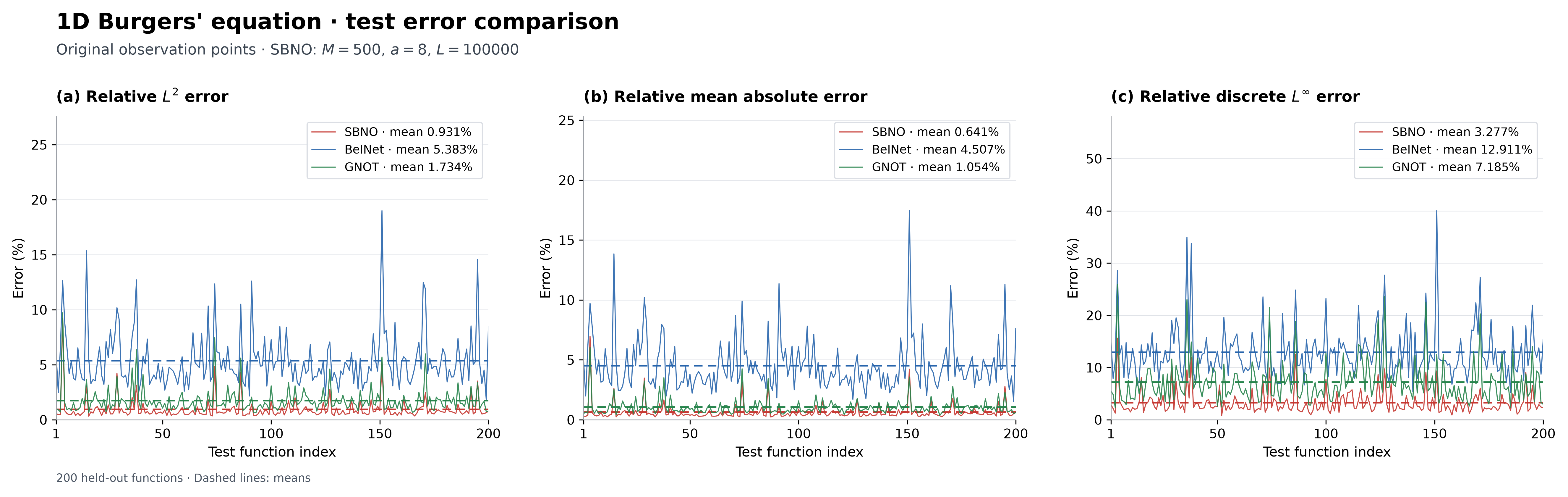} 
    \caption{2D Darcy Flow SBNO vs GNOT vs BelNet}
\end{figure}

\subsection{Five-dimensional nonlinear elliptic equations}
Finally, we consider the two equations
\begin{equation}\label{exp:nonlinear5d}
    -\Delta u + u +u^3=f,
    \qquad\text{and}\qquad
    -\Delta u+e^u+u^3=f,
    \qquad x\in\mathbb{R}^5,
\end{equation}
with $u(x)\to0$ as $|x|\to\infty$.
We observe the functions on $[-1,1]^5$ and learn the map $f\mapsto u$ from
the sampled pairs $\{f(x_i),u(x_i)\}_{i=1}^{N}$.
The boundary values of the restricted solutions are inherited from their
whole-space definitions and are generally nonzero on the boundary of the
observation cube. These are stationary problems, so no initial condition is
prescribed.

To generate the data, we first choose a solution
\begin{equation}\label{exp:manufactured5d}
    u^{(k)}(x)=\sum_{\beta\in S_k}c^k_\beta h^5_\beta(x),
    \qquad S_k\subset\{\beta\in\mathbb{N}_0^5:|\beta|\leq6\},
    \qquad |S_k|=50,
\end{equation}
where $h^5_\beta$ denotes a normalized Hermite function with parameter $a=5$.
For each function, we select the $50$ indices uniformly without replacement
from the $2002$ available modes and draw the coefficients independently from
$\mathcal{N}(0,1/50)$.
We then evaluate $f^{(k)}=-\Delta u^{(k)}+ u^{(k)} + (u^{(k)})^3$ or
$f^{(k)}=-\Delta u^{(k)}+e^{u^{(k)}}+(u^{(k)})^3$ analytically at uniformly sampled points
in $[-1,1]^5$. Both nonlinear terms are evaluated without truncation.
This gives $2000$ training pairs and $1000$ independent testing pairs for each
equation. Only the sampled function values and coordinates are supplied to
the models; the Hermite coefficients are used to generate the data.
For the exponential equation, $f(x)\to1$ at infinity, while $f-1$ and $u$
are square-integrable on $\mathbb{R}^5$. The observations and reported errors
are restricted to the cube. Note that the second PDE is a challenging one: once we take the exponential nonlinearity of $ u $, the input $ f $ encodes infinitely Gauss-Hermitian modes. 

We plot the relative $ L^{2} $ test error, the $ \mathbb{E}_{X}\left[ \frac{\lvert \hat{u}(X) - u(X) \rvert}{\lvert u(X) \rvert} \right]$ error, and the relative $ L^{\infty} $ error for SBNO in Figure 4. We use a total of $ 1 \times 10^{5} $ real samplings for both $ -\Delta u + u + u^{3} = f $ and $ -\Delta u + e^{u} + u^{3} = f $. Figure 4 demonstrates that SBNO can achieve a very good testing error for relatively simple nonlinearity, and can still learn the solution operator even when highly nonlinear term $ e^{u} $. We expect better numerical results for better designs of neural networks. These results suggest that our proposed method has great potential to be applied to more challenging PDEs and to higher spatial dimensions.
\begin{figure}[h]
\includegraphics[width=165mm, height=100mm]{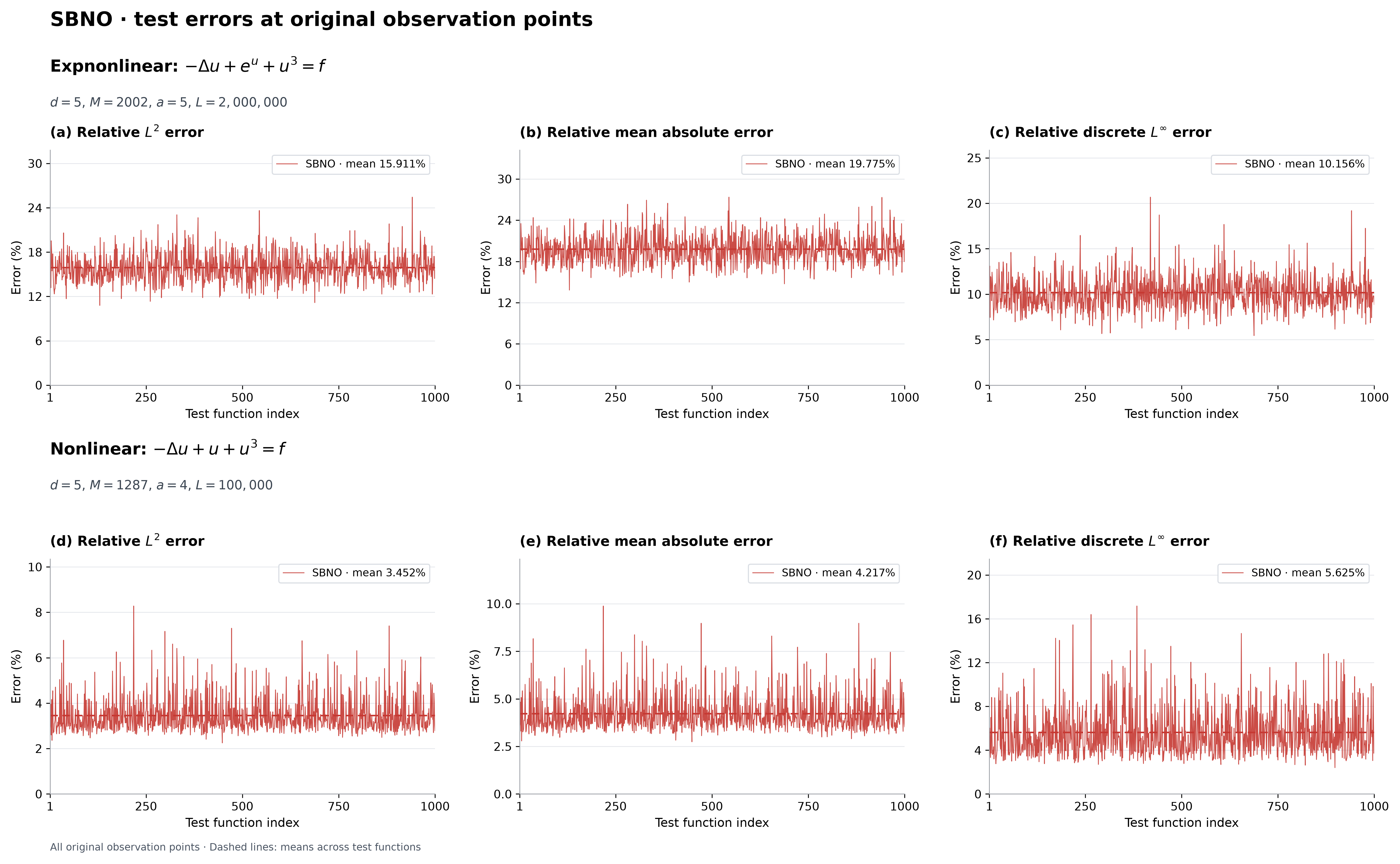} 
    \caption{5D Two Nonlinear PDEs SBNO}
\end{figure}
\medskip

\noindent {\bf{Data Availability Statement and Conflict of Interest.}} \\
No external datasets were used in this study. All data used in the numerical experiments were generated computationally using the procedures described in this paper. The authors declare no conflicts of interest related to this work.
\medskip

\section{Appendix: Validity of Assumption B}
In this section, we show that with probability $ 1- \delta $, where $ \delta $ can be aribtrarily small, the {\it{Assumption B}} with $ \eta_{l} \equiv 0 $ for all $ l $ holds with some assumptions that will be stated below. It indicates that the {\it{Assumption B}} is a general and reasonable assumption. The proof follows the construction of \cite[Theorem 3.2]{DU}.

\subsection{Model and statement}

Fix observation locations $x_1,\ldots,x_N$ and a training set

\[
\{(\mathbf f^{(k)},y_k)\}_{k=1}^K,\qquad
\mathbf f^{(k)}=(f^{(k)}(x_1),\ldots,f^{(k)}(x_N))^\top\in\mathbb R^N,
\quad \|\mathbf f^{(k)}\|_2=1,\quad |y_k|\le C_y.
\tag{1}
\]

The training set does not change with $L$. Let $A(z,x)$ be fixed; in the intended application, $A=A_{a,M}$ with fixed $a$ and $M$. Draw $Z_1,\ldots,Z_L$ independently from a fixed distribution $\rho$, and hold them fixed during training. Set

\[
v_l(W,Z_l)=
 \bigl(w_{l1}A(Z_l,x_1),\ldots,w_{lN}A(Z_l,x_N)\bigr)^\top.
\tag{2}
\]

To use the real ReLU with a complex kernel, take the real part of the preactivation:

\[
\begin{aligned}
u_k(W)&=\frac1{\sqrt L}\sum_{l=1}^L b_l
 \sigma\!\left(\operatorname{Re}
       (v_l(W,Z_l)^\top\mathbf f^{(k)})\right),\\
\sigma(s)&=\max\{s,0\},\qquad
\mathcal L(W)=\frac12\sum_{k=1}^K(u_k(W)-y_k)^2.
\end{aligned}
\tag{3}
\]

Write $w_l=p_l+iq_l$ and $\theta_l=(p_l,q_l)\in\mathbb R^{2N}$. Initialize $\theta_l(0)\sim\mathcal N(0,I_{2N})$ and $b_l\sim\operatorname{Unif}\{-1,1\}$, independently across neurons and independently of the $Z_l$. Keep $b_l$ fixed and train only $\theta_l$ by Euclidean gradient flow for $\mathcal L$, using the ReLU differential-inclusion convention. Thus the trained parameters are the original $w_{li}$, through their real and imaginary parts.

For the proof, define

\[
\xi_k(z)=
 \begin{pmatrix}
  (\operatorname{Re}A(z,x_i)\,\mathbf f_i^{(k)})_{i=1}^N\\
  (-\operatorname{Im}A(z,x_i)\,\mathbf f_i^{(k)})_{i=1}^N
 \end{pmatrix},
\qquad B(z)=\max_{1\le i\le N}|A(z,x_i)|.
\tag{4}
\]

Then $\operatorname{Re}(v_l^\top\mathbf f^{(k)})=\theta_l^\top\xi_k(Z_l)$ and $\|\xi_k(z)\|_2\le B(z)$. The vectors $\xi_k(Z_l)$ are auxiliary expressions for differentiating (3); the training inputs remain $\mathbf f^{(k)}$.

Define

\[
H^\infty_{kj}
 =\mathbb E_{Z,G}\left[
 \xi_k(Z)^\top\xi_j(Z)\,
 \mathbf1_{\{G^\top\xi_k(Z)>0,\ G^\top\xi_j(Z)>0\}}\right],
\quad
G\sim\mathcal N(0,I_{2N}),\quad G\perp Z.
\tag{5}
\]

\begin{theorem}
Suppose
\[
\mu_4:=\mathbb E_\rho B(Z)^4<\infty,
\qquad \lambda_0:=\lambda_{\min}(H^\infty)>0.
\tag{6}
\]

Write $\mu_j=\mathbb E_\rho B(Z)^j$ for $j=2,3$ and $S=C_y^2+\mu_2/2$. There is a universal constant $C>0$ such that, for any $\delta\in(0,1)$, if

\[
L\ge C\left(
 \frac{K^2\mu_4}{\lambda_0^2\delta}
 +\frac{K^6\mu_3^2S}{\lambda_0^4\delta^3}
 \right),
\tag{7}
\]

then, with probability at least $1-\delta$ over complex sampling and initialization,

\[
\|u(t)-y\|_2^2
 \le e^{-\lambda_0t}\|u(0)-y\|_2^2
 \quad\text{for all }t\ge0.
\tag{8}
\]

In particular, $\mathcal L(W(t))\to0$ as $t\to\infty$ for each such fixed finite width. The event can be chosen to work for every global admissible gradient-flow trajectory.
\end{theorem}
\begin{proof}
We follow the four steps in \S3 of \cite{DU}: concentration at initialization, stability near initialization, control of weight movement, and continuation for all time.

Write $\xi_{kl}=\xi_k(Z_l)$ and let $s_{kl}(t)$ be the ReLU derivative at $\theta_l(t)^\top\xi_{kl}$, with a value in $[0,1]$ permitted at zero. Almost everywhere,

\[
\begin{aligned}
\dot\theta_l
 &=\frac{b_l}{\sqrt L}\sum_{k=1}^K(y_k-u_k)s_{kl}\xi_{kl},\\
\dot u&=H(t)(y-u),\qquad
H_{kj}(t)=\frac1L\sum_{l=1}^L
 \xi_{kl}^\top\xi_{jl}s_{kl}(t)s_{jl}(t).
\end{aligned}
\tag{9}
\]

This is the Gram matrix for training $W$ itself, with

\[
\xi_k(z)^\top\xi_j(z)
 =\sum_{i=1}^N|A(z,x_i)|^2\mathbf f_i^{(k)}\mathbf f_i^{(j)}.
\]

The summands of $H(0)$ are independent across neurons, have mean $H^\infty$, and their entries have second moments at most $\mu_4$. Thus

\[
\mathbb E\|H(0)-H^\infty\|_F^2\le\frac{K^2\mu_4}{L}.
\tag{10}
\]

Markov's inequality and the first term of (7) give, with probability at least $1-\delta/3$,

\[
\|H(0)-H^\infty\|_2\le\lambda_0/4,\qquad
\lambda_{\min}(H(0))\ge3\lambda_0/4.
\tag{11}
\]

Put $B_l=B(Z_l)$ and $c_0=\sqrt{2/\pi}$. For $R>0$, define

\[
E_{kl}(R)=
 \left\{\xi_{kl}\ne0,\ 
 |\theta_l(0)^\top\xi_{kl}|
 \le R B_l\|\xi_{kl}\|_2\right\}.
\]

Conditionally on $Z_l$, the normalized nonzero Gaussian projection has variance one, so

\[
\mathbb P(E_{kl}(R)\mid Z_l)\le c_0 R B_l.
\tag{12}
\]

If $\|\theta_l-\theta_l(0)\|_2\le R B_l$, the activation on input $k$ can change only on $E_{kl}(R)$. When $\xi_{kl}=0$, its Gram contribution is zero. Uniformly over this neighborhood and over derivative choices at zero,

\[
\|H-H(0)\|_2
 \le T_R:=
 \frac1L\sum_{l=1}^L\sum_{k,j=1}^K
 B_l^2\bigl(\mathbf1_{E_{kl}(R)}+\mathbf1_{E_{jl}(R)}\bigr).
\]

By (12), $\mathbb E T_R\le2c_0K^2R\mu_3$. Choose

\[
R=\frac{\delta\lambda_0}{24c_0K^2\mu_3}.
\tag{13}
\]

Then $\mathbb P(T_R>\lambda_0/4)\le\delta/3$. Together with (11), this gives $\lambda_{\min}(H)\ge\lambda_0/2$ throughout the neighborhood. The radii $R B_l$ explicitly retain the dependence on $Z_l$.

The independent output signs give $\mathbb E u_k(0)=0$, and the Gaussian ReLU identity gives

\[
\mathbb E u_k(0)^2
 =\tfrac12\mathbb E\|\xi_k(Z)\|_2^2\le\mu_2/2.
\]

Therefore $\mathbb E\|u(0)-y\|_2^2\le KS$, and with probability at least $1-\delta/3$,

\[
E_0^2:=\|u(0)-y\|_2^2\le\frac{3KS}{\delta}.
\tag{14}
\]

As long as the trajectory remains in the neighborhood, (9) implies

\[
\frac{d}{dt}\|u-y\|_2^2
 =-2(u-y)^\top H(t)(u-y)
 \le-\lambda_0\|u-y\|_2^2.
\]

Hence $\|u(t)-y\|_2\le e^{-\lambda_0t/2}E_0$ and

\[
\begin{aligned}
\|\theta_l(t)-\theta_l(0)\|_2
 &\le\frac{B_l\sqrt K}{\sqrt L}
       \int_0^t\|u(s)-y\|_2\,ds\\
 &\le B_l R',\qquad
R'=\frac{2\sqrt K\,E_0}{\lambda_0\sqrt L}.
\end{aligned}
\tag{15}
\]

Equations (7), (13), and (14), with a sufficiently large universal $C$, give $R'\le R/2$.

The three events above hold jointly with probability at least $1-\delta$. On their intersection, there can be no first exit from the neighborhood: up to any proposed exit time, (15) bounds each displacement by half its allowed radius. Neurons with $B_l=0$ have zero velocity. The bounded displacements also allow continuation for all time. The Gram lower bound therefore holds globally, and integration of the loss differential inequality proves (8).
\end{proof}

\subsection{Application to the truncated Segal--Bargmann kernel}

\begin{corollary}
Let
\[
A(z,x)=A_{a,M}(z,x)
 =\sum_{\beta\in\Lambda_M}e_\beta^a(z)h_\beta^a(x),
\qquad
d\rho_a(z)=\frac{\sum_{\beta\in\Lambda_M}|e_\beta^a(z)|^2}{M}\,d\mu_a(z),
\tag{16}
\]

with fixed $a>0$ and finite $\Lambda_M$ containing $0$. If no two distinct training vectors $\mathbf f^{(k)}$ are parallel, both conditions in (6) hold, so (7) and (8) apply.
\end{corollary}
\begin{proof}
For each sensor, $A_{a,M}(z,x_i)$ is a polynomial in $z$, while $\rho_a$ has a polynomial-times-Gaussian density. Thus $B(Z)$ has finite moments of every finite order.

The constant coefficient of $A_{a,M}(z,x_i)$ is $h_0^a(x_i)>0$. Each polynomial is nonzero, and its zero set has Lebesgue measure zero. Absolute continuity of $\rho_a$ and finiteness of $N$ imply that all $A_{a,M}(Z,x_i)$ are nonzero almost surely.

For each such $z$, the real-linear map $\mathbf f\mapsto\xi(z,\mathbf f)$ is injective and preserves nonparallelness. Apply \cite[Theorem 3.1]{DU} to the normalized vectors $\xi_k(z)/\|\xi_k(z)\|_2$. Multiplication by their positive norms shows that the conditional Gaussian Gram matrix on the unnormalized $\xi_k(z)$ is positive definite. Its expectation is $H^\infty$ and is finite by the moment bound. Every nonzero vector has strictly positive conditional quadratic form almost surely, and hence strictly positive expected quadratic form. Since $K$ is finite, $\lambda_{\min}(H^\infty)>0$.

\end{proof}

\bibliographystyle{plain}
\bibliography{SegalBargmann}
\end{document}